\documentclass{aims}
\usepackage{amsmath}
\usepackage{paralist}
\usepackage{graphics}
\usepackage{epsfig}
\usepackage{graphicx}
\usepackage{epstopdf}
\usepackage[colorlinks=true]{hyperref}
\hypersetup{urlcolor=blue, citecolor=red}
\usepackage{tikz}
\newcommand*\circled[1]{\tikz[baseline=(char.base)]{\node[shape=circle,draw,inner sep=1pt] (char) {\small #1};}}

\def\currentvolume{X}
\def\currentissue{X}
\def\currentyear{200X}
\def\currentmonth{XX}
\def\ppages{X--XX}
\def\DOI{10.3934/xx.xx.xx.xx}

\newtheorem{theorem}{Theorem}
\newtheorem{corollary}[theorem]{Corollary}
\newtheorem{lemma}[theorem]{Lemma}
\newtheorem{proposition}[theorem]{Proposition}

\theoremstyle{definition}

\newcommand{\C}{{\mathsf C}}
\newcommand{\R}{{\mathbb R}}
\renewcommand{\S}{{\mathbb S}}
\newcommand{\N}{{\mathbb N}}
\renewcommand{\L}{\mathrm L}

\newcommand{\be}[1]{\begin{equation}\label{#1}}
\newcommand{\ee}{\end{equation}}
\renewcommand{\(}{\left(}
\renewcommand{\)}{\right)}

\newcommand{\irdmu}[1]{\int_{\R^d}{#1}\,d\mu}

\newcommand{\nrm}[2]{\|{#1}\|_{\L^{#2}(\R^d)}}
\newcommand{\iRd}[1]{\int_{\R^d}{#1}\,dx}

\newcommand{\scal}[2]{\left\langle{#1},{#2}\right\rangle}

\newcommand{\ird}[1]{\int_{\R^d}{#1}\;dx}

\newcommand{\M}{\mathfrak M}
\newcommand{\iM}[1]{\int_{\M}{#1}\,d\kern1pt v_g}

\newcommand{\irdsph}[3]{\int_0^\infty\kern-5pt\int_{\S^{d-1}}#2\;{#1}^{\kern1pt #3}\,\frac{d{#1}}{{#1}}\;d\omega}
\newcommand{\D}[1]{\mathsf D_\alpha\kern1pt#1}

\newcommand{\sd}{\sigma_d}
\newcommand{\rate}{\rho}

\def\qed{\,\unskip\kern 6pt \penalty 500 \raise -2pt\hbox{\vrule \vbox to8pt{\hrule width 6pt \vfill\hrule}\vrule}\par}
\def\cprime{$'$}
\newcommand{\scaling}{\mu}
\newcommand{\fg}{f}
\newcommand{\ellk}{k}
\newcommand{\kell}{\ell}

\definecolor{darkblue}{rgb}{0.05, .05, .65}
\definecolor{darkgreen}{rgb}{0.1, .65, .1}
\definecolor{darkred}{rgb}{0.8,0,0}

\begin{document}
\title[Weighted fast diffusion I]{Weighted fast diffusion equations (Part I):\\ Sharp asymptotic rates without symmetry and symmetry breaking in Caffarelli-Kohn-Nirenberg inequalities}
\author[M.~Bonforte, J.~Dolbeault, M.~Muratori and B.~Nazaret]{}

\date{\today}


\email{matteo.bonforte@uam.es}
\email{dolbeaul@ceremade.dauphine.fr}
\email{matteo.muratori@unipv.it}
\email{bruno.nazaret@univ-paris1.fr}

\thanks{$^*$ Corresponding author.}

%%%%%%%%%%%%%%%%%%%%%%%%%%%%%%%%%%%%%%%%%%%%%%%%%%%%%%%%%%%%%%%%%%%%%%
\maketitle
\thispagestyle{empty}
\vspace*{-0.45cm}

\centerline{\scshape Matteo Bonforte}
\smallskip
{\footnotesize
\centerline{Departamento de Matem\'{a}ticas,}
\centerline{Universidad Aut\'{o}noma de Madrid,}
\centerline{Campus de Cantoblanco, 28049 Madrid, Spain}
}\medskip

\centerline{\scshape Jean Dolbeault $^*$}
\smallskip
{\footnotesize
\centerline{Ceremade, UMR CNRS nr.~7534,}
\centerline{Universit\'e Paris-Dauphine, PSL Research University,}
\centerline{Place de Lattre de Tassigny, 75775 Paris Cedex~16, France}
}\medskip

\centerline{\scshape Matteo Muratori}
\smallskip
{\footnotesize
\centerline{Dipartimento di Matematica \emph{Felice Casorati},}
\centerline{Universit\`a degli Studi di Pavia,}
\centerline{Via A.~Ferrata 5, 27100 Pavia, Italy}
}\medskip

\centerline{\scshape Bruno Nazaret}
\smallskip
{\footnotesize
\centerline{SAMM,}
\centerline{Universit\'e Paris 1,}
\centerline{90, rue de Tolbiac, 75634 Paris Cedex~13, France}
}

%%%%%%%%%%%%%%%%%%%%%%%%%%%%%%%%%%%%%%%%%%%%%%%%%%%%%%%%%%%%%%%%%%%%%%
\begin{abstract}\vspace*{-0.25cm} In this paper we consider a family of Caffarelli-Kohn-Nirenberg interpolation inequalities (CKN), with two radial power law weights and exponents in a subcritical range. We address the question of \emph{symmetry breaking:} are the optimal functions radially symmetric, or not ? Our intuition comes from a \emph{weighted fast diffusion} (WFD) flow: if symmetry holds, then an explicit \emph{entropy -- entropy production inequality} which governs the \emph{intermediate asymptotics} is indeed equivalent to (CKN), and the self-similar profiles are optimal for (CKN).

We establish an explicit symmetry breaking condition by proving the linear instability of the radial optimal functions for (CKN). Symmetry breaking in (CKN) also has consequences on entropy -- entropy production inequalities and on the intermediate asymptotics for (WFD). Even when no symmetry holds in (CKN), \emph{asymptotic rates of convergence} of the solutions to (WFD) are determined by a weighted Hardy-Poincar\'e inequality which is interpreted as a linearized entropy -- entropy production inequality. All our results rely on the study of the bottom of the spectrum of the linearized diffusion operator around the self-similar profiles, which is equivalent to the linearization of (CKN) around the radial optimal functions, and on variational methods. Consequences for the (WFD) flow will be studied in Part~II of this work.\end{abstract}

%%%%%%%%%%%%%%%%%%%%%%%%%%%%%%%%%%%%%%%%%%%%%%%%%%%%%%%%%%%%%%%%%%%%%%
%%%%%%%%%%%%%%%%%%%%%%%%%%%%%%%%%%%%%%%%%%%%%%%%%%%%%%%%%%%%%%%%%%%%%%
\section{Introduction and main results}\label{Sec:Intro}

Let us consider the \emph{fast diffusion equation with weights}
\be{FD}
u_t+|x|^\gamma\,\nabla\cdot\(\,|x|^{-\beta}\,u\,\nabla u^{m-1}\)=0\,,\quad(t,x)\in\R^+\times\R^d\,,
\ee
where $\beta$ and $\gamma$ are two real parameters, and $m\in[m_1,1)$ with
\[
m_1:=\tfrac{2\,d-2-\beta-\gamma}{2\,(d-\gamma)}\,.
\]
Equation~\eqref{FD} admits self-similar solutions
\[
u_\star(t,x)=\(t/\rate\)^{-\,\rate\,(d-\gamma)}\,\mathfrak B_{\beta,\gamma}\big((\rate/t)^\rho\,x\big)\,,\quad\forall\,(t,x)\in\R^+\times\R^d\,,
\]
where $1/\rate=(d-\gamma)\,(m-m_c)$ with $m_c:=\tfrac{d-2-\beta}{d-\gamma}$ and, up to a multiplication by a constant and a scaling,
\[
\mathfrak B_{\beta,\gamma}(x)=\(1+|x|^{2+\beta-\gamma}\)^\frac1{m-1}\quad\forall\,x\in\R^d\,.
\]
Such self-similar solutions are generalizations of \emph{Barenblatt self-similar solutions} which are known to govern the asymptotic behavior of the solutions of~\eqref{FD} as $t\to+\infty$ when $(\beta,\gamma)=(0,0)$. In that case, optimal rates of convergence have been determined in uniform norms or by relative entropy methods in~\cite{MR586735,MR1940370,MR1777035,MR1986060}. However, when $(\beta,\gamma)\neq(0,0)$, the analysis is more delicate because of possible \emph{symmetry breaking} issues.

Assume for a while that \emph{symmetry} holds (this assumption will be made precise below). Then a rate of convergence of the solutions to~\eqref{FD} towards $u_\star$, known in the literature as the problem of \emph{intermediate asymptotics}, is bounded in terms of the \emph{best constant} $\C_{\beta,\gamma,p}$ in the \emph{Caffarelli-Kohn-Nirenberg interpolation inequalities}
\be{CKN}
\nrm w{2p,\gamma}\le\C_{\beta,\gamma,p}\,\nrm{\nabla w}{2,\beta}^\vartheta\,\nrm w{p+1,\gamma}^{1-\vartheta}\quad\forall\,w\in C_0^\infty(\R^d)\,.
\ee
These inequalities have been introduced in~\cite{Caffarelli-Kohn-Nirenberg-84}. Here $C_0^\infty(\R^d)$ denotes the space of smooth functions on $\R^d$ which converge to zero as $|x|\to\infty$, $m$ and $p$ are related by
\[
p=\tfrac1{2\,m-1}\quad\Longleftrightarrow\quad m=\tfrac{p+1}{2\,p}\,,
\]
the parameters $\beta$, $\gamma$ and $p$ are subject to the restrictions
\be{parameters}
d\ge2\,,\quad\gamma-2<\beta<\tfrac{d-2}d\,\gamma\,,\quad\gamma\in(-\infty,d)\,,\quad p\in\(1,p_\star\right]\quad\mbox{with}\quad p_\star:=\tfrac{d-\gamma}{d-2-\beta}\,,
\ee
and the exponent $\vartheta$ is determined by the scaling invariance, \emph{i.e.},
\be{theta}
\vartheta=\tfrac{(d-\gamma)\,(p-1)}{p\,(d+2+\beta-2\,\gamma-p\,(d-2-\beta))}\,.
\ee
The norms involved in~\eqref{CKN} are defined by
\[
\nrm w{q,\gamma}:=\(\iRd{|w|^q\,|x|^{-\gamma}}\)^{1/q}\quad\mbox{and}\quad\nrm wq:=\nrm w{q,0}\,.
\]
We also define the space $\L^{q,\gamma}(\R^d)$ as the space of all measurable functions $w$ such that $\nrm w{q,\gamma}$ is finite. A simple density argument shows that~\eqref{CKN} can be extended with no restriction to the space of the functions $w\in\L^{p+1,\gamma}(\R^d)$ such that $\nabla w\in\L^{2,\beta}(\R^d)$. See Section~\ref{Sec:CKNrange} for further considerations on the functional setting.

Because of the weights, it is not straightforward to decide whether optimality in~\eqref{CKN} is achieved by radial functions, or not. For some values of the parameters there is a competition between the weights which tend to decenter the optimizer and the nonlinearity for which radial functions are in principle preferable. The main result of this paper is that weights win over the nonlinearity for certain values of $\beta$ and $\gamma$, hence proving a \emph{symmetry breaking} result that can be precisely characterized as follows. Let us consider the subset $C_{0,\mathrm{rad}}^\infty(\R^d)$ of radial functions in~$C_0^\infty(\R^d)$ and the reduced interpolation inequalities
\be{CKNrad}
\nrm w{2p,\gamma}\le\C_{\beta,\gamma,p}^\star\,\nrm{\nabla w}{2,\beta}^\vartheta\,\nrm w{p+1,\gamma}^{1-\vartheta}\quad\forall\,w\in C_{0,\mathrm{rad}}^\infty(\R^d)\,.
\ee
Since $\C_{\beta,\gamma,p}$ is the best constant in~\eqref{CKN} \emph{without symmetry assumption}, the symmetry breaking issue is the question of knowing whether equality (symmetry case) holds in the inequality $\C_{\beta,\gamma,p}\ge\C_{\beta,\gamma,p}^\star$, or not (symmetry breaking case). As we shall see later, the equality case in~\eqref{CKNrad} is achieved by
\[
w_\star(x)=\mathfrak B_{\beta,\gamma}^{m-1/2}(x)=\(1+|x|^{2+\beta-\gamma}\)^\frac1{p-1}\quad\forall\,x\in\R^d\,,
\]
which provides us with an explicit expression of $\C_{\beta,\gamma,p}^\star$: see Appendix~\ref{Appendix:Mass}.

The limit case $p=p_\star$, that is, $\vartheta=1$ and $\beta=d-2-(d-\gamma)/p$, corresponds to the critical case in~\eqref{CKN} and the symmetry breaking issue has been fully solved in~\cite{DEL2015}. In particular, $\beta=\gamma-2$ can be achieved only in the limit as $p=p_\star=1$ in which~\eqref{CKN} degenerates into a Hardy type inequality for which $\C_{\gamma-2,\gamma,1}=\C_{\gamma-2,\gamma,1}^\star$, but admits no minimizers with gradient in $\L^{2,\gamma-2}(\R^d)$. The other threshold case $\beta=\tfrac{d-2}d\,\gamma$ is also covered by our results, except when $d=2$, in which case one has to assume that $\beta<0$ and $p\in(1,p_\star)$ with $p_\star=+\infty$. To avoid lengthy statements, we will ignore it in the rest of this paper, but necessary adaptations are straightforward.

\medskip The equality case in~\eqref{CKNrad} is achieved not only by $w_\star$ but also by $w=u_\star^{m-1/2}(t,\cdot)$, for any $t>0$, because~the inequality is homogenous and scale invariant. This is the first relation between the evolution equation~\eqref{FD} and the inequality~\eqref{CKN}. Now let us come back to the question of the \emph{intermediate asymptotics}. At a formal level, we observe that a solution to~\eqref{FD} with nonnegative initial datum $u_0\in\L^1(\R^d,|x|^{-\gamma}\,dx)$ is such that
\[
\frac d{dt}\int_{\R^d}u\,\frac{dx}{|x|^\gamma}=0\,,
\]
which suggests to introduce the time-dependent rescaling
\be{Eqn:TDRS}
u(t,x)=R^{\gamma-d}\,v\((2+\beta-\gamma)^{-1}\,\log R,\frac xR\)
\ee
with $R=R(t)$ defined by
\[
\frac{dR}{dt}=(2+\beta-\gamma)\,R^{(m-1)(\gamma-d)-(2+\beta-\gamma)+1}\,,\quad R(0)=1\,.
\]
This ordinary differential equation can be solved explicitly and we obtain that
\[
R(t)=\(1+\tfrac{2+\beta-\gamma}\rate\,t\)^\rate
\]
with $1/\rate=(1-m)\,(\gamma-d)+2+\beta-\gamma=(d-\gamma)\,(m-m_c)$. The equation for $v$ is of Fokker-Planck type and takes the form
\be{Eqn:FD-FP}
v_t+|x|^\gamma\,\nabla\cdot\Big[\,|x|^{-\beta}\,v\,\nabla\big(v^{m-1}-|x|^{2+\beta-\gamma}\big)\Big]=0
\ee
with initial condition $v(t=0,\cdot)=u_0$. Barenblatt type stationary solutions are given by
\[
\mathfrak B(x)=\(C_M+|x|^{2+\beta-\gamma}\)^\frac1{m-1}
\]
where $C_M$ is uniquely determined by the condition
\[
\int_{\R^d}\mathfrak B\;\frac{dx}{|x|^\gamma}=M:=\int_{\R^d}u_0\;\frac{dx}{|x|^\gamma}\,.
\]
Since the mass can be fixed arbitrarily using the scaling properties of~\eqref{FD} and unless it is explicitly specified, we make the choice that $M=M_\star$ is such that $C_{M_\star}=1$. See Appendix~\ref{Appendix:Mass} for an expression of $M_\star$. To emphasize the dependence of~$\mathfrak B$ in the parameters $\beta$ and $\gamma$, we shall write it $\mathfrak B_{\beta,\gamma}$ consistently with our previous notations.

When symmetry holds so that $\C_{\beta,\gamma,p}=\C_{\beta,\gamma,p}^\star$, Inequality~\eqref{CKN} can be written as an \emph{entropy -- entropy production} inequality
\be{Ineq:E-EP}
\tfrac{1-m}m\,(2+\beta-\gamma)^2\,\mathcal F[v]\le\mathcal I[v]\,,
\ee
and equality is achieved by $\mathfrak B_{\beta,\gamma}$. Here the \emph{free energy} (which is sometimes called \emph{generalized relative entropy} in the literature) and the \emph{relative Fisher information} are defined respectively by
\[
\mathcal F[v]:=\frac1{m-1}\int_{\R^d}\(v^m-\mathfrak B_{\beta,\gamma}^m-m\,\mathfrak B_{\beta,\gamma}^{m-1}\,(v-\mathfrak B_{\beta,\gamma})\)\,\frac{dx}{|x|^\gamma}
\]
and
\[
\mathcal I[v]:=\int_{\R^d}v\left|\,\nabla v^{m-1}-\nabla\mathfrak B_{\beta,\gamma}^{m-1}\right|^2\,\frac{dx}{|x|^\beta}\,.
\]
The equivalence of~\eqref{CKN} and~\eqref{Ineq:E-EP} will be detailed in Section~\ref{Sec:NSIIneqEP}. However, we do not claim that $\tfrac{1-m}m\,(2+\beta-\gamma)^2$ is the optimal constant in the entropy -- entropy production inequality, and this is in general not the case: we refer to Section~\ref{Sec:Conclusion} for a discussion of this issue.

By evolving the free energy along the flow and differentiating with respect to $t$, we obtain
\[
\frac d{dt}\mathcal F[v(t,\cdot)]=-\,\frac m{1-m}\,\mathcal I[v(t,\cdot)]\,,
\]
which provides us with a first result.
%---------------------------------------------------------------------
\begin{proposition}\label{Prop:GlobalRate} Assume that the parameters satisfy~\eqref{parameters}, let $m=\frac{p+1}{2\,p}$ and consider a solution to~\eqref{FD} with nonnegative initial datum $u_0\in\L^{1,\gamma}(\R^d)$ such that $\nrm{u_0^m}{1,\gamma}$ and $\int_{\R^d}u_0\,|x|^{2+\beta-2\gamma}\,dx$ are finite. Then the function $v$ given in terms of $u$ by~\eqref{Eqn:TDRS} solves~\eqref{Eqn:FD-FP} and we have that
\be{EntropyDecay}
\mathcal F[v(t,\cdot)]\le\mathcal F[u_0]\,e^{-(2+\beta-\gamma)^2t}\quad\forall\,t\ge0
\ee
if one of the following two conditions is satisfied:
\begin{enumerate}
\item [(i)] either $u_0$ is a.e.~radially symmetric,
\item [(ii)] or symmetry holds in~\eqref{CKN}.
\end{enumerate}\end{proposition}
%---------------------------------------------------------------------
The condition under which symmetry holds has been established after this paper was submitted, in~\cite{DELM2015}, and will be commented below. Under the conditions (i) or~(ii), Inequalities~\eqref{Ineq:E-EP} and~\eqref{EntropyDecay} are actually equivalent as can be shown by computing $\frac d{dt}\mathcal F[v(t,\cdot)]$ at $t=0$. On the other hand,~\eqref{EntropyDecay} gives a strong control on the large time asymptotics. In terms of the rescaled function~$v$, as in~\cite{MR1940370}, one can prove using an adapted Csisz\'ar-Kullback-Pinsker inequality that
\[
\nrm{v(t,\cdot)-\mathfrak B_{\beta,\gamma}}{1,\gamma}^2\le\mathsf C_{\rm CKP}(M)\,\mathcal F[u_0]\,e^{-\,(2+\beta-\gamma)^2\,t}\quad\forall\,t\ge0
\]
for some explicit constant $\mathsf C_{\rm CKP}(M)$. If we replace $v(t,\cdot)$ by $\mu^{\gamma-d}(t)\,v(t,\cdot/\mu(t))$, as in~\cite{1004,MR3103175,1751-8121-48-6-065206,1501}, one can even obtain a faster convergence rate for some function~$\mu$ such that $\lim_{t\to+\infty}\mu(t)=1$. After undoing the change of variables~\eqref{Eqn:TDRS}, this provides us with an algebraic rate of convergence in original variables. Proofs and more details can be found in~\cite{BDMN2016b}.

\medskip The results of Proposition~\ref{Prop:GlobalRate} hold only for radial solutions to~\eqref{FD}, or under the assumption that symmetry holds, \emph{i.e.}, if $\C_{\beta,\gamma,p}=\C_{\beta,\gamma,p}^\star$, and our first main result is a negative result, in the sense that it gives us a sufficient condition on $\beta$ and $\gamma$ under which symmetry breaking holds, \emph{i.e.}, for which $\C_{\beta,\gamma,p}>\C_{\beta,\gamma,p}^\star$.

Let us define
\[
\beta_{\rm FS}(\gamma):=d-2-\sqrt{(d-\gamma)^2-4\,(d-1)}\,.
\]
%---------------------------------------------------------------------
\begin{theorem}\label{Thm:SymmetryBreaking} Assume that the parameters satisfy~\eqref{parameters}. Then \emph{symmetry breaking} holds in~\eqref{CKN} if
\[
\gamma<0\quad\mbox{and}\quad\beta_{\rm FS}(\gamma)<\beta<\frac{d-2}d\,\gamma\,.
\]
\end{theorem}
%---------------------------------------------------------------------
The symmetry breaking region is shown in Fig.~\ref{Fig:F1}. In~\cite{DELM2015}, it has been proved that symmetry holds if $0\le\gamma\le d$, or $\gamma<0$ and $\beta\le\beta_{\rm FS}(\gamma)$, which is the complementary domain, in the range of admissible parameters, of the symmetry breaking region.
%---------------------------------------------------------------------
\begin{figure}[ht]
\begin{center}
\includegraphics[width=10cm]{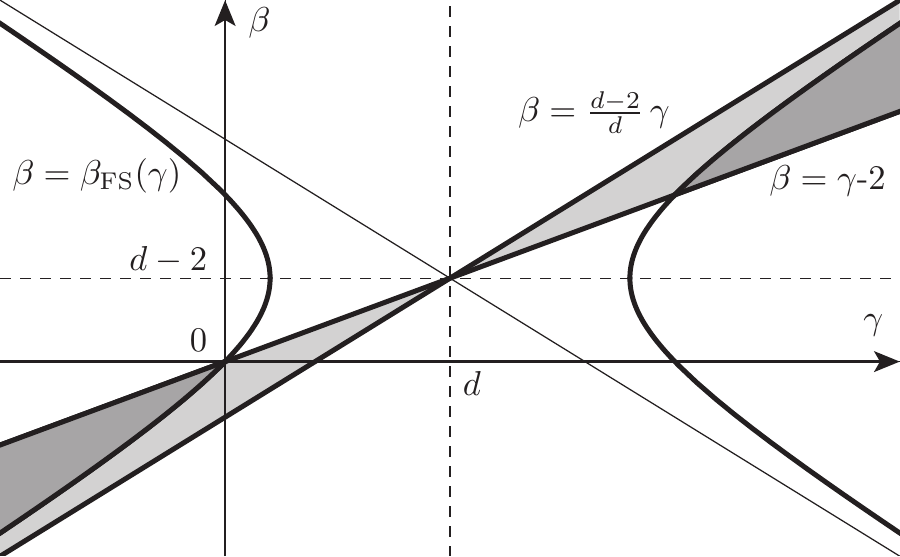}
\caption{\label{Fig:F1} We consider the admissible range for the $(\beta,\gamma)$ parameters. The grey area is the area of validity of~\eqref{CKN} and it is given by $\gamma-2<\beta<\tfrac{d-2}d\,\gamma$ if $\gamma<d$ and $\tfrac{d-2}d\,\gamma<\beta<\gamma-2$ if $\gamma>d$: the cones corresponding to $\gamma<d$ and $\gamma>d$ are in one-to-one correspondance by an \emph{inversion symmetry}: see details in Section~\ref{Sec:CKNrange}. Notice that the case $\gamma>d$ has been excluded in~\eqref{parameters} in order to simplify the statements. The hyperbola defined by the \emph{Felli \& Schneider curve} determines a region (dark grey area) of symmetry breaking which is valid for any $p\in(1,p_\star)$, and independent of $p$. However, since $p_\star$ depends on $\beta$ and $\gamma$, this induces an additional restriction on the admissible range of $(\beta,\gamma)$, which depends on $p$: see Figs.~\ref{Fig:F2} and~\ref{Fig:F3}. Here we consider the special case $d=5$.}
\end{center}
\end{figure}
%---------------------------------------------------------------------
It is a remarkable fact that $\beta_{\rm FS}$ is independent of $p$. Here `FS' stands for V.~Felli and M.~Schneider, who first gave the  sharp condition of linear instability for symmetry breaking in the critical case $p= p_\star$: see Section~\ref{Sec:CKNrange} for details. Notice that the condition $p\le p_\star$ can be seen as a restriction on the admissible set of parameters $(\beta,\gamma)$. For any given $p\in(1,\frac d{d-2}]$, it means
\[
\beta\ge d-2-\frac{d-\gamma}p\,.
\]
As $p\to\frac d{d-2}$, the admissible cone corresponding to $\gamma<d$ shrinks to the simple half-line given by $\beta=\frac{d-2}d\,\gamma$, while the whole range of~\eqref{parameters} is covered in the limit as $p\to1$. See Figs.~\ref{Fig:F2} and~\ref{Fig:F3}.

The proof of Theorem~\ref{Thm:SymmetryBreaking} relies on the linear instability of optimal radial functions, among non-radial functions. Our purpose is \emph{not} to study the symmetry issue in the general Caffarelli-Kohn-Nirenberg interpolation inequalities, which is a difficult problem that has to be dealt with using specific methods: see~\cite{DELM2015}. However, even without taking the symmetry issue in~\eqref{CKN} into account, we can study the asymptotic rates of convergence. Since Barenblatt type profiles attract all solutions  at least when $m\in[m_1,1)$, the linearization around these profiles is again enough to get an answer. This is the purpose of our second main result. Better results concerning the basin of attraction of the Barenblatt profiles are stated in part II of this paper: see~\cite{BDMN2016b}. For technical reasons and in order to simplify the proof, we shall assume that the initial datum $u_0$ is sandwiched between two Barenblatt profiles: \emph{there are two positive constants $C_1$ and $C_2$ such that}
\be{Ineq:sandwiched}
\(C_1+|x|^{2+\beta-\gamma}\)^\frac1{m-1}\le u_0(x)\le\(C_2+|x|^{2+\beta-\gamma}\)^\frac1{m-1}\quad\forall\,x\in\R^d\,.
\ee
Let us define
\[
\sigma(\gamma,p)=-\,\frac{\big(d-\gamma+p\,(d+2-\gamma)\big)\,\big(d-\gamma-p\,(d-2+\gamma)\big)}{2\,p\,(p+1)\,(d-\gamma)}
\]
and consider the unique positive solution to
\be{Eqn:eq-eta}
\eta\,(\eta+n-2)=\frac{d-1}{\alpha^2}.
\ee
where $n$ and $\alpha$ are defined by
\be{Eqn:alpha-n}
\alpha=1+\frac{\beta-\gamma}2\quad\mbox{and}\quad n=2\,\frac{d-\gamma}{2+\beta-\gamma}\,.
\ee
 See Figs.~\ref{Fig:F2} and~\ref{Fig:F3} for an illustration of the curve $\beta=\sigma(\gamma,p)$. Hence $\eta$ is given by
\be{Eqn:eta}
\eta=\sqrt{\tfrac{d-1}{\alpha^2}+\big(\tfrac{n-2}2\big)^2}-\tfrac{n-2}2=\tfrac2{2+\beta-\gamma}\sqrt{d-1+\big(\tfrac{d-2-\beta}2\big)^2}-\tfrac{d-2-\beta}{2+\beta-\gamma}\,.
\ee
%---------------------------------------------------------------------
\begin{theorem}\label{Thm:Asymptotic rates} {\rm~\cite{BDMN2016b}} Assume that~\eqref{parameters} and~\eqref{Ineq:sandwiched} hold. With $m=\frac{p+1}{2\,p}$ and $\eta$ given by~\eqref{Eqn:eta}, if $v$ solves~\eqref{Eqn:FD-FP}, then there exists a positive constant $C$ such that
\[\label{AsymptoticEntropyDecay}
\mathcal F[v(t,\cdot)]\le C\,e^{-\,2\,\lambda\,t}\quad\forall\,t\ge0
\]
where $\lambda$ is given by
\[
\lambda=\left\{\begin{array}{rl}
\frac{2+\beta-\gamma}{2\,p}\,\big[d-\gamma-p\,(d+\gamma-2\,\beta-4)\big]\quad&\mbox{if}\quad\beta\le\sigma(\gamma,p)\,,\\[6pt]
\tfrac12\,(2+\beta-\gamma)^2\,\eta\quad&\mbox{if}\quad\beta\ge\sigma(\gamma,p)\,.
\end{array}
\right.
\]
\end{theorem}
%---------------------------------------------------------------------
The constant $C$ depends non-explicitly on $u_0$. The condition~\eqref{Ineq:sandwiched} may look rather restrictive, but it is probably not, because it is expected that the condition is satisfied, for some positive $t$, by any solution with initial datum as in~Proposition~\ref{Prop:GlobalRate}. At least this is what occurs when $(\beta,\gamma)=(0,0)$: see for instance~\cite{BV-JFA,BBDGV}. Since our purpose is only to investigate the large time behavior of the solutions, establishing such a regularization result is definitely out of the scope of the present paper.

Theorem~\ref{Thm:Asymptotic rates} is proved in~\cite{BDMN2016b}, with less restrictions on $m$. A detailed study of the regularity of the solutions to~\eqref{FD} is indeed required. In this paper, we will give a first proof of Theorem~\ref{Thm:SymmetryBreaking} which relies on the result of Theorem~\ref{Thm:Asymptotic rates} because this completes the picture of the relation of the evolution problem~\eqref{FD} with the question of symmetry breaking in~\eqref{CKN} and emphasizes the role of the spectrum of the linearized problems. For completeness, in Section~\ref{Sec:SBVar}, we will also give a purely variational proof of Theorem~\ref{Thm:SymmetryBreaking} which does not use the results of Theorem~\ref{Thm:Asymptotic rates}.

In practice, $\lambda$ is the optimal asymptotic rate and it is given by the relation
\[
\lambda=(1-m)\,\Lambda
\]
where $\Lambda$ is the optimal constant in a spectral gap inequality, or \emph{Hardy-Poincar\'e inequality}, which goes as follows. With $\delta=1/(1-m)$, let us define $d\mu_\delta:=\mu_\delta(x)\,dx$,
\[
\mu_\delta(x):=\frac1{(1+|x|^2)^\delta}
\]
and $\D v:=\(\alpha\,\frac{\partial v}{\partial s},\frac1{s}\,\nabla_{\kern-2pt\omega} v\)$, where $\alpha=1+(\beta-\gamma)/2$. Here~$s=|x|$ and $\omega=x/|x|$ are spherical coordinates. We shall also use the parameter $n=(d-\gamma)/\alpha$ as in~\eqref{Eqn:alpha-n}.
%---------------------------------------------------------------------
\begin{proposition}\label{Prop:SpectralGap} Let $d\ge2$, $\alpha\in(0,+\infty)$, $n>d$ and $\delta\ge n$. Then the \emph{Hardy-Poincar\'e inequality}
\be{Hardy-Poincare1}
\irdmu{|\D f|^2\,|x|^{n-d}}_\delta\ge\Lambda\irdmu{|f|^2\,|x|^{n-d}}_{\delta+1}
\ee
holds for any $f\in\L^2(\R^d,|x|^{n-d}\,d\mu_{\delta+1})$ such that $\irdmu{f\,|x|^{n-d}}_{\delta+1}=0$, with an optimal constant $\Lambda$ given by
\be{Eqn:SpectralGap}
\Lambda=\left\{\begin{array}{rl}
2\,\alpha^2\,(2\,\delta-n)\quad&\mbox{if}\quad0<\alpha^2\le\frac{(d-1)\,\delta^2}{n\,(2\,\delta-n)\,(\delta-1)}\,,\\[6pt]
2\,\alpha^2\,\delta\,\eta\quad&\mbox{if}\quad\alpha^2>\frac{(d-1)\,\delta^2}{n\,(2\,\delta-n)\,(\delta-1)}\,,
\end{array}
\right.
\ee
where $\eta$ is given by~\eqref{Eqn:eta}.\end{proposition}
%---------------------------------------------------------------------
The two possible values of $\Lambda$ simply mean that $\Lambda=\min\{\Lambda_{0,1},\Lambda_{1,0}\}$ where $\Lambda_{1,0}=2\,\alpha^2\,(2\,\delta-n)$ and $\Lambda_{0,1}=2\,\alpha^2\,\delta\,\eta$ are respectively the lowest positive eigenvalue among radial functions, and the lowest positive eigenvalue among non-radial functions. The case $\alpha^2=\frac{(d-1)\,\delta^2}{n\,(2\,\delta-n)\,(\delta-1)}$ corresponds to the threshold case for which $\Lambda_{1,0}=\Lambda_{0,1}$ and is reflected in Theorem~\ref{Thm:Asymptotic rates} by the case $\beta=\sigma(\gamma,p)$. The condition $\delta\ge n$ comes from the sub-criticality condition $p\le p_\star$ and can be replaced by $\delta>0$ in Proposition~\ref{Prop:SpectralGap}. Under appropriate conditions on $u_0$, Theorem~\ref{Thm:Asymptotic rates} can also be extended to the strict super-critical range corresponding to $m<m_1$: see \cite{BDMN2016b}.

\medskip The outline of the paper goes as follows. Section~\ref{Sec:CKN} is devoted to Caffarelli-Kohn-Nirenberg inequalities~\eqref{CKN} from a variational point of view. Considerations on the weighted fast diffusion equation and the free energy estimates have been collected in Section~\ref{Sec:WFD}. These considerations are formal but constitute the guideline of our strategy. The proofs of the results involving the nonlinear flow are given in~\cite{BDMN2016b}. Spectral results on the linearized evolution operator and the associated quadratic forms are established in Section~\ref{Sec:SpectrumResults}, which also contains the proof of our main results. The key technical result is Lemma~\ref{Lem:SpectrumResults}, but some additional spectral results have been collected in Appendix~\ref{Sec:SpectrumAditional}.

\medskip Let us conclude this introduction by a brief overview of the literature. Concerning fundamental results on Caffarelli-Kohn-Nirenberg inequalities, we primarily refer to~\cite{Catrina-Wang-01}. When $p=p_\star$, the symmetry condition found by V.~Felli and M.~Schneider in~\cite{Felli-Schneider-03} has recently been proved to be optimal in~\cite{DEL2015}. The interested reader is invited to refer to this last paper for a rather complete list of earlier results. Still concerning symmetry breaking issues in Caffarelli-Kohn-Nirenberg inequalities, one can quote~\cite{DDFT,DETT,1406}, and~\cite{1005} for some associated existence results. We have no specific references for~\eqref{CKN} with general parameters $\beta$, $\gamma$ and $p$ apart the original paper~\cite{Caffarelli-Kohn-Nirenberg-84} by L.~Caffarelli R.~Kohn and L.~Nirenberg, and to our knowledge, no symmetry breaking result was known so far for~\eqref{CKN} apart when $p=p_\star$. However, when $\beta=0$, one has to quote~\cite{DMN2015} in which existence and symmetry for $\gamma>0$ but small is established, and~\cite{DELM2015} for recent, complete symmetry results. Notice in particular that~\eqref{FD} is considered in~\cite{DMN2015} together with entropy methods when $\beta=0$, and plays an important role in the heuristics of the method used in~\cite{DELM2015}.

References concerning global existence for equations related with~\eqref{FD}, large time behavior of the solutions and intermediate asymptotics will be listed in~\cite{BDMN2016b}. Here let us only mention some papers dealing with \emph{linearizations} of non-weighted fast diffusion equations. Roughly speaking, we can distinguish three categories of papers: 1) some early results based mostly on comparison methods: see~\cite{MR0097235,MR586735,MR956056,MR1491842} and references therein; 2) a linearization motivated by the gradient flow structure of the fast diffusion equations:~\cite{DKM,MR1982656,DenzMc05,MR2246356}; 3) entropy based approaches:~\cite{MR2320246,BBDGV,BDGV-PNAS,BGV,MR2344717,MR1853037,MR1901093,MR1793019,MR1777035,MR2355628,MR1986060,1004,MR3103175,1751-8121-48-6-065206,1501,MR1974458}. This last angle of attack is the one of this paper and many more references can be found in the above mentioned papers. The reader interested in a historical perspective on entropy methods can refer to~\cite{MR1842428} and to the review article~\cite{MR2065020}. Let us quote~\cite{MR2510011} for related issues in probability theory.

Beyond the interest for the understanding of qualitative issues like symmetry breaking in functional inequalities, Equation~\eqref{FD} is motivated by some applications which are listed in~\cite{BDMN2016b}. From a more abstract point of view, let us emphasize that power law weights and power nonlinearities are typical of the asymptotic analysis of some limiting regimes -- either at large scales or close to eventual singularities -- which are obtained by rescalings and blow-up methods. Thus the study of~\eqref{FD} and~\eqref{CKN} can be considered as an important theoretical issue for a large class of applied problems.

%%%%%%%%%%%%%%%%%%%%%%%%%%%%%%%%%%%%%%%%%%%%%%%%%%%%%%%%%%%%%%%%%%%%%%
%%%%%%%%%%%%%%%%%%%%%%%%%%%%%%%%%%%%%%%%%%%%%%%%%%%%%%%%%%%%%%%%%%%%%%
\section{Caffarelli-Kohn-Nirenberg inequalities: a variational point of view}\label{Sec:CKN}

%%%%%%%%%%%%%%%%%%%%%%%%%%%%%%%%%%%%%%%%%%%%%%%%%%%%%%%%%%%%%%%%%%%%%%
\subsection{Range of the parameters and symmetry breaking region}\label{Sec:CKNrange}

In their simplest form, the Caffarelli-Kohn-Nirenberg inequalities
\be{CKN-DEL}
\(\ird{\frac{|v|^q}{|x|^{bq}}}\)^{2/q}\le\,\mathcal C_{a,b}\ird{\frac{|\nabla v|^2}{|x|^{2a}}}\quad\forall\,v\in\mathcal D_{a,b}
\ee
have been established in~\cite{Caffarelli-Kohn-Nirenberg-84}, under the conditions that $a\le b\le a+1$ if $d\ge3$, $a<b\le a+1$ if $d=2$, $a+1/2<b\le a+1$ if $d=1$, and $a<a_c$ where
\[
a_c:=\frac{d-2}2\,.
\]
The exponent
\[\label{exponent-relationabp}
q=\frac{2\,d}{d-2+2\,(b-a)}
\]
is determined by the invariance of the inequality under scalings. Here $\mathcal C_{a,b}$ denotes the optimal constant in~\eqref{CKN-DEL} and the space $\mathcal D_{a,b}$ defined by
\[
\mathcal D_{a,b}:=\Big\{\,w\in\L^q\(\R^d,|x|^{-bq}\,dx\)\,:\,|x|^{-a}\,|\nabla w|\in\L^2(\R^d,dx)\Big\}
\]
is obtained as the completion of $C_c^\infty(\R^d)$, the space of smooth functions in $\R^d$ with compact support, with respect to the norm defined by $\|w\|^2=\|\,|x|^{-b}\,w\,\|_q^2+\|\,|x|^{-a}\,\nabla w\,\|_2^2$. Inequality~\eqref{CKN} holds also for $a>a_c$: in this case $\mathcal D_{a,b}$ has to be defined as the completion with respect to $\|\cdot\|$ of the space $C_c^\infty(\R^d\setminus\{0\}):=\big\{w\in C_c^\infty(\R^d)\,:\,\mbox{supp}(w)\subset\R^d\setminus\{0\}\big\}$. The two cases, $a>a_c$ and $a<a_c$, are related by the property of \emph{modified inversion symmetry} that can be found in~\cite[Theorem~1.4,~(ii)]{Catrina-Wang-01}. In the setting of Inequality~\eqref{CKN}, this property becomes a simpler \emph{inversion symmetry} property that will be discussed below. We refer to~\cite{Catrina-Wang-01} for many important properties of~\eqref{CKN-DEL}, to~\cite{Felli-Schneider-03} and~\cite{DEL2015} respectively for a symmetry breaking condition and for symmetry results.

Inequality~\eqref{CKN} enters in the framework of the Caffarelli-Kohn-Nirenberg inequalities introduced in~\cite{Caffarelli-Kohn-Nirenberg-84}. However, these inequalities are easy to justify directly from~\eqref{CKN-DEL}. Indeed, by a H\"older interpolation, we see that
\[
\nrm w{2p,\gamma}\le\nrm w{2p_\star,\gamma}^\vartheta\,\nrm w{p+1,\gamma}^{1-\vartheta}
\]
with $p_\star=\tfrac{d-\gamma}{d-2-\beta}$, with same expression as in~\eqref{parameters} and $\vartheta=\tfrac{(d-\gamma)\,(p-1)}{p\,(d+2+\beta-2\,\gamma-p\,(d-2-\beta))}$ as in~\eqref{theta}. With the choice
\[
2\,p_\star=q\,,\quad a=\frac\beta2\quad\mbox{and}\quad b=\frac\gamma q\,,
\]
the reader is invited to check that~\eqref{CKN} follows from~\eqref{CKN-DEL} with an optimal constant $\C_{\beta,\gamma,p}\le\mathcal C_{a,b}^\vartheta$. The range $a<b<a+1<a_c+1$ is transformed into the range
\[
\gamma-2<\beta<\tfrac{d-2}d\,\gamma\quad\mbox{and}\quad\gamma<d
\]
and guarantees that~\eqref{CKN} holds true for any $p\in(1,p_\star)$ and $d\ge2$. Finally let us notice that $\vartheta=1$ if $p=p_\star$, in which case~\eqref{CKN} is actually reduced to~\eqref{CKN-DEL}.

The \emph{inversion symmetry} property of~\eqref{CKN} can be stated as follows. The admissible range of parameters corresponding to $\gamma<d$ and the one corresponding to $\gamma>d$ are in one-to-one correspondance. With
\[
\widetilde\beta=2\,(d-2)-\beta\quad\mbox{and}\quad\widetilde\gamma=2\,d-\gamma\,,
\]
the inequality for a function $w$ in the range $\gamma-2\le\beta\le\frac{d-2}d\,\gamma<d-2$ is equivalent to the inequality for
\[
\widetilde w(x)=w\big(\tfrac x{|x|^2}\big)\quad\forall\,x\in\R^d\setminus\{0\}
\]
in the range $d-2<\frac{d-2}d\,\widetilde\gamma\le\widetilde\beta\le\widetilde\gamma-2$ because
\begin{multline*}
\C_{\beta,\gamma,p}\,\nrm{\nabla\widetilde w}{2,\widetilde\beta}^\vartheta\,\nrm{\widetilde w}{p+1,\widetilde\gamma}^{1-\vartheta}=\C_{\beta,\gamma,p}\,\nrm{\nabla w}{2,\beta}^\vartheta\,\nrm w{p+1,\gamma}^{1-\vartheta}\\
\ge\nrm w{2p,\gamma}=\nrm{\widetilde w}{2p,\widetilde\gamma}\,.
\end{multline*}
Since $\vartheta=\vartheta(\beta,\gamma)$ as defined in~\eqref{parameters} is such that
\[
\vartheta(\beta,\gamma)=\vartheta(\widetilde\beta,\widetilde\gamma)\,,
\]
we conclude that
\[
\C_{\beta,\gamma,p}=\C_{\widetilde\beta,\widetilde\gamma,p}\,.
\]
Hence, in this paper, as it is usual in the study of Caffarelli-Kohn-Nirenberg inequalities, we consider only the cases $a<a_c$ and $\gamma<d$. 

As noticed in the introduction, a remarkable property is the fact that the symmetry breaking condition for~\eqref{CKN} written in Theorem~\ref{Thm:SymmetryBreaking} amounts in terms of the parameters $\alpha$ and $n$ defined in~\eqref{Eqn:alpha-n} to
\be{FS}
\alpha>\sqrt{\frac{d-1}{n-1}}\,,
\ee
which does not depend on $p$ and coincides with the sharp symmetry breaking condition of V.~Felli \& M.~Schneider for~\eqref{CKN-DEL}, as found in~\cite{Felli-Schneider-03,DEL2015}. In terms of the parameters of~\eqref{CKN-DEL}, this condition is usually stated as
\[
a<0\quad\mbox{and}\quad b<b_{\,\rm FS}(a):=\frac{d\,(a_c-a)}{2\sqrt{(a_c-a)^2+d-1}}+a-a_c\,.
\]

The \emph{Felli \& Schneider curve} is given by the set of equations
\[
n=\frac{d-2-\beta}\alpha+2=\frac{d-\gamma}\alpha\quad\mbox{and}\quad\alpha^2=\frac{d-1}{n-1}\,.
\]
Altogether, we find that~\eqref{FS} means
\[
(2+\beta-\gamma)\,(2\,d-\beta-\gamma-2)>4\,(d-1)\,,
\]
that is,
\[
h_{\rm FS}(\beta,\gamma):=(d-\gamma)^2-(\beta-d+2)^2-4\,(d-1)<0\,.
\]
The threshold of this domain is given by the hyperbola $\gamma\mapsto\big(\gamma,\beta_\pm(\gamma)\big)$ and corresponds, in the admissible parameter range with $\gamma<d$ to $\beta=\beta_{\rm FS}(\gamma)$, in the language of Theorem~\ref{Thm:SymmetryBreaking}. See Fig.~\ref{Fig:F1}.

%%%%%%%%%%%%%%%%%%%%%%%%%%%%%%%%%%%%%%%%%%%%%%%%%%%%%%%%%%%%%%%%%%%%%%
\subsection{An existence result}\label{Sec:ExistenceCKN}
%---------------------------------------------------------------------
\begin{proposition}\label{Prop:Existence} Assume that $\beta$, $\gamma$ and $p$ satisfy~\eqref{parameters}. Then there exists an optimal function in $\L^{p+1,\gamma}(\R^d)$ with $\nabla w\in\L^{2,\beta}(\R^d)$ such that
\[
\nrm w{2p,\gamma}=\C_{\beta,\gamma,p}\,\nrm{\nabla w}{2,\beta}^\vartheta\,\nrm w{p+1,\gamma}^{1-\vartheta}\,,
\]
where $\C_{\beta,\gamma,p}$ is the best constant in~\eqref{CKN}.\end{proposition}
%---------------------------------------------------------------------
\begin{proof} The proof is similar to the one of~\cite[Proposition~2.5]{DMN2015} when $\beta=0$. Details are left to the reader.\end{proof}

%%%%%%%%%%%%%%%%%%%%%%%%%%%%%%%%%%%%%%%%%%%%%%%%%%%%%%%%%%%%%%%%%%%%%%
\subsection{From Caffarelli-Kohn-Nirenberg to Gagliardo-Nirenberg inequalities}\label{Sec:ChangeOfVariables}

Written in spherical coordinates for a function
\[
u(r,\omega)=w(x)\,,\quad\mbox{with}\quad r=|x|\quad\mbox{and}\quad\omega=\frac x{|x|}\,,
\]
Inequality~\eqref{CKN} becomes
\begin{multline*}
\(\irdsph r{|u|^{2p}}{d-\gamma}\)^\frac1{2p}\\
\le\C_{\beta,\gamma,p}\(\irdsph r{\left|\nabla u\right|^2}{d-\beta}\)^\frac\vartheta2\(\irdsph r{|u|^{p+1}}{d-\gamma}\)^\frac{1-\vartheta}{p+1}
\end{multline*}
where $\left|\nabla u\right|^2=\left|\tfrac{\partial u}{\partial r}\right|^2+\tfrac1{r^2}\,\left|\nabla_{\kern-2pt\omega} u\right|^2$ and $\nabla_{\kern-2pt\omega} u$ denotes the gradient of $u$ with respect to the angular variable $\omega\in\S^{d-1}$. Next we consider the change of variables $r\mapsto s=r^\alpha$,
\[\label{wv}
u(r,\omega)=v(r^\alpha,\omega)\quad\forall\,(r,\omega)\in\R^+\times\S^{d-1}
\]
so that
\begin{multline*}
\(\irdsph s{|v|^{2p}}{\frac{d-\gamma}\alpha}\)^\frac1{2p}\\
\le\alpha^{-\zeta}\,\C_{\beta,\gamma,p}\(\irdsph s{\(\alpha^2\left|\tfrac{\partial v}{\partial s}\right|^2+\tfrac1{s^2}\,|\nabla_{\kern-2pt\omega} v|^2\)}{\frac{d-2-\beta}\alpha+2}\)^\frac\vartheta2\\
\times\(\irdsph s{|v|^{p+1}}{\frac{d-\gamma}\alpha}\)^\frac{1-\vartheta}{p+1}
\end{multline*}
with
\be{zeta}
\zeta:=\frac\vartheta2+\frac{1-\vartheta}{p+1}-\frac1{2\,p}=\frac{(2+\beta-\gamma)\,(p-1)}{2\,p\,\big(d+2+\beta-2\,\gamma-p\,(d-2-\beta)\big)}\,.
\ee
We pick $\alpha$ so that
\[
n=\frac{d-2-\beta}\alpha+2=\frac{d-\gamma}\alpha\,.
\]
Solving these equations means that $n$ and $\alpha$ are given by~\eqref{Eqn:alpha-n}. The change of variables $s=r^\alpha$ is therefore responsible for the introduction of the two parameters, $\alpha$ and $n$, which were involved for instance in the statement of Proposition~\ref{Prop:SpectralGap} and in the discussion of the symmetry breaking region in Section~\ref{Sec:CKNrange}. With the notation
\[
\D v=\(\alpha\,\frac{\partial v}{\partial s},\frac1{s}\,\nabla_{\kern-2pt\omega} v\)\,,
\]
we can write a first inequality,
\begin{multline*}
\(\irdsph s{|v|^{2p}}n\)^\frac1{2p}\\
\le\alpha^{-\zeta}\,\C_{\beta,\gamma,p}\(\irdsph s{|\D v|^2}n\)^\frac\vartheta2\(\irdsph s{|v|^{p+1}}n\)^\frac{1-\vartheta}{p+1}\,,
\end{multline*}
which is equivalent to~\eqref{CKN}. From the point of view of its scaling properties, this inequality is an analogue of a Gagliardo-Nirenberg in dimension~$n$, at least when $n$ is an integer, but the variable $\omega$ still belongs to a sphere of dimension $d-1$. The parameter $\alpha$ is a measure of the intensity of the derivative in the radial direction compared to angular derivatives and plays a crucial role in the symmetry breaking issues, as shown by Condition~\eqref{FS}. Let us summarize what we have shown so far.
%---------------------------------------------------------------------
\begin{proposition}\label{Prop:GNweigthed} Assume that $\alpha>0$, $n>d$ and $p\in(1,\frac n{n-2}]$. Then the following inequality holds
\be{CKN1}
\nrm v{2p,d-n}\le\mathsf K_{\alpha,n,p}\,\nrm{\D v}{2,d-n}^\vartheta\,\nrm v{p+1,d-n}^{1-\vartheta}\quad\forall\,v\in C_0^\infty(\R^d)\,.
\ee
The optimal constant $\mathsf K_{\alpha,n,p}$ is related with the optimal constant in~\eqref{CKN} by
\[
\C_{\beta,\gamma,p}=\alpha^\zeta\,\mathsf K_{\alpha,n,p}\,,
\]
with $\zeta$ given by~\eqref{zeta}. When symmetry holds, equality in~\eqref{CKN1} is achieved by the function
\[
x\mapsto v_\star(x):=(1+|x|^2)^\frac1{1-p}\,.
\]
\end{proposition}
%---------------------------------------------------------------------
For brevity, we shall refer to~\eqref{CKN1} as a \emph{weighted Gagliardo-Nirenberg inequality}.

%%%%%%%%%%%%%%%%%%%%%%%%%%%%%%%%%%%%%%%%%%%%%%%%%%%%%%%%%%%%%%%%%%%%%%
\subsection{A linear stability analysis}\label{Sec:VarCKN-LinearStab}

Let us define the functional
\begin{multline*}
\mathcal F[v]:=\vartheta\,\log\(\nrm{\D v}{2,d-n}\)+(1-\vartheta)\,\log\(\nrm v{p+1,d-n}\)+\log\mathsf K_{\alpha,n,p}\\
-\log\(\nrm v{2p,d-n}\)
\end{multline*}
obtained by taking the difference of the logarithm of the two terms in~\eqref{CKN1}. Since~$v_\star$ is a critical point of $\mathcal F$, a Taylor expansion of $\mathcal F[v_\star+\varepsilon\,g]$ at order $\varepsilon^2$ shows that
\[
\mathcal F[v_\star+\varepsilon\,g]=\varepsilon^2\,\mathcal Q[g]+o(\varepsilon^2)
\]
with
\[
\frac2{\mathsf a}\,\mathcal Q[g]=\nrm{\D g}{2,d-n}^2+\frac{\mathsf b}{\mathsf a}\,\iRd{|g|^2\,\frac{|x|^{n-d}}{1+|x|^2}}-\frac{\mathsf c}{\mathsf a}\,\iRd{|g|^2\,\frac{|x|^{n-d}}{\(1+|x|^2\)^2}}
\]
and
\begin{eqnarray*}
&&\mathsf a=\vartheta\,\nrm{\D v_\star}{2,d-n}^{-2}\,,\\
&&\mathsf b=p\,(1-\vartheta)\,\nrm{v_\star}{p+1,d-n}^{-(p+1)}\,,\\
&&\mathsf c=(2\,p-1)\,\nrm{v_\star}{2p,d-n}^{-2\,p}\,.
\end{eqnarray*}
With $\mathsf v(s)=\(1+s^2\)^{-\frac1{p-1}}$, let us compute $\mathsf a$, $\mathsf b$ and $\mathsf c$ using
\begin{eqnarray*}
&&\mathsf A:=\int_0^\infty |\mathsf v'|^2\,s^{n-1}\,ds=\frac4{(p-1)^2}\int_0^\infty\(1+s^2\)^{-\frac{2\,p}{p-1}}\,s^{n+1}\,ds\,,\\
&&\mathsf B:=\int_0^\infty |\mathsf v|^{p+1}\,s^{n-1}\,ds=\int_0^\infty\(1+s^2\)^{-\frac{p+1}{p-1}}\,s^{n-1}\,ds\,,\\
&&\mathsf C:=\int_0^\infty |\mathsf v|^{2\,p}\,s^{n-1}\,ds=\int_0^\infty\(1+s^2\)^{-\frac{2\,p}{p-1}}\,s^{n-1}\,ds\,.\\
\end{eqnarray*}
We may observe that
\[
\tfrac14\,(p-1)^2\,\mathsf A=\int_0^\infty\(1+s^2-1\)\(1+s^2\)^{-\frac{2\,p}{p-1}}\,s^{n-1}\,ds=\mathsf B-\mathsf C
\]
and
\begin{multline*}
\mathsf B=\int_0^\infty\(1+s^2\)\(1+s^2\)^{-\frac{2\,p}{p-1}}\,s^{n-1}\,ds\\
=\mathsf C-\tfrac12\,\tfrac{p-1}{p+1}\int_0^\infty\frac d{ds}\left[\(1+s^2\)^{-\frac{p+1}{p-1}}\right]\,s^n\,ds=\mathsf C+\tfrac n2\,\tfrac{p-1}{p+1}\,\mathsf B\,.
\end{multline*}
Altogether, this proves that
\[
\frac{\mathsf A}{\mathsf B}=\frac{2\,n}{p^2-1}\quad\mbox{and}\quad\frac{\mathsf A}{\mathsf C}=\frac{4\,n}{p-1}\,\frac1{n+2-p\,(n-2)}
\]
and
\begin{multline*}
\frac2{\mathsf a}\,\mathcal Q[g]=\nrm{\D g}{2,d-n}^2+p\,\frac{1-\vartheta}\vartheta\,\frac{\mathsf A}{\mathsf B}\,\alpha^2\iRd{|g|^2\,\frac{|x|^{n-d}}{1+|x|^2}}\\
-\frac{2\,p-1}\vartheta\,\frac{\mathsf A}{\mathsf C}\,\alpha^2\iRd{|g|^2\,\frac{|x|^{n-d}}{\(1+|x|^2\)^2}}\,.
\end{multline*}
Replacing in terms of the original parameters and once all computations are done, we find that the quadratic form $\mathcal Q$ has to be considered on the space $\mathcal X$ of the functions
\[
g\in\L^2\big(\R^d,\tfrac{|x|^{n-d}}{1+|x|^2}\,dx\big)\quad\mbox{such that}\quad\iRd{g\,\tfrac{|x|^{n-d}}{1+|x|^2}}=0\,.
\]
and it is such that
\begin{multline*}
\frac2{\mathsf a}\,\mathcal Q[g]=\nrm{\D g}{2,d-n}^2+\frac{p\,(2+\beta-\gamma)}{(p-1)^2}\,\big[d-\gamma-p\,(d-2-\beta)\big]\iRd{|g|^2\,\frac{|x|^{n-d}}{1+|x|^2}}\\
-p\,(2\,p-1)\,\frac{(2+\beta-\gamma)^2}{(p-1)^2}\iRd{|g|^2\,\frac{|x|^{n-d}}{\(1+|x|^2\)^2}}\,.
\end{multline*}
If $\mathcal Q$ takes negative values, this means that the minimum of $\mathcal F$ cannot be achieved by $v_\star$. In other words, the existence of a function $g$ such that $\mathcal Q[g]<0$ would prove the linear instability of $\mathcal F$ at the critical point $v_\star$. This question will be studied in Section~\ref{Sec:SpectrumResults}.

%%%%%%%%%%%%%%%%%%%%%%%%%%%%%%%%%%%%%%%%%%%%%%%%%%%%%%%%%%%%%%%%%%%%%%
%%%%%%%%%%%%%%%%%%%%%%%%%%%%%%%%%%%%%%%%%%%%%%%%%%%%%%%%%%%%%%%%%%%%%%
\section{The weighted fast diffusion equation}\label{Sec:WFD}

In this section, we develop a formal approach, which will be fully justified in~\cite{BDMN2016b}. Heuristically, this section is essential to understand the role of the evolution equation~\eqref{FD}. Let us start with the entropy -- entropy production inequality, which governs the global convergence rates.

%%%%%%%%%%%%%%%%%%%%%%%%%%%%%%%%%%%%%%%%%%%%%%%%%%%%%%%%%%%%%%%%%%%%%%
\subsection{The equivalence of the Caffarelli-Kohn-Nirenberg and of an entropy -- entropy production inequality in the symmetry range}\label{Sec:NSIIneqEP}

Let us consider Inequality~\eqref{CKN1}. Up to a scaling, the determination of the best constant $\mathsf K_{\alpha,n,p}$ is equivalent to the minimization of the functional
\[
v\mapsto\mathcal H[v]:=\tfrac{\mathrm A}2\,\nrm{\D v}{2,d-n}^2+\tfrac{\mathrm B}{p+1}\,\nrm v{p+1,d-n}^{p+1}-\tfrac1{2\,p}\,\nrm v{2p,d-n}^{2\,p\,\frac{n+2-p(n-2)}{n-p(n-4)}}
\]
where the two positive constants $\mathrm A$ and $\mathrm B$ are chosen such that
\[
v_\star(x)=\(1+|x|^2\)^{-\frac1{p-1}}\quad\forall\,x\in\R^d
\]
is a critical point with critical level $0$.
%---------------------------------------------------------------------
\begin{proposition}\label{Prop:EquivEP-NSII} Assume that the parameters satisfy~\eqref{parameters}. With the notations of Section~\ref{Sec:Intro} and $(\alpha,n)$ defined by~\eqref{Eqn:alpha-n}, for any $u\in C_0^\infty(\R^d)$ such that $\nrm u{1,\gamma}=M_\star$ and $v$ such that $u^{m-\frac12}(x)=v\big(|x|^{\alpha-1}\,x\big)$ for any $x\in\R^d$, we have
\[
\mathcal I[u]-\tfrac{1-m}m\,(2+\beta-\gamma)^2\,\mathcal F[u]=\frac4\alpha\,\frac{(m-1)^2}{(2\,m-1)^2}\,\mathcal H[v]\,.
\]
\end{proposition}
%---------------------------------------------------------------------
See~\cite{MR1940370} for details in a similar result, without weights. The consequences of the presence of weights will be discussed in Section~\ref{Sec:Conclusion}.
\begin{proof} Let us give the main steps of the computation. If we expand the free energy and the Fisher information, we obtain that
\[
\mathcal F[u]:=\frac1{m-1}\int_{\R^d}u^m\,\frac{dx}{|x|^\gamma}+\frac m{1-m}\int_{\R^d}|x|^{2+\beta-2\gamma}\,(u-\mathfrak B_{\beta,\gamma})\,dx-\frac{\nrm{\mathfrak B_{\beta,\gamma}^m}{1,\gamma}}{m-1}
\]
and
\begin{multline*}
\mathcal I[u]:=4\,\frac{(m-1)^2}{(2\,m-1)^2}\int_{\R^d}\big|\,\nabla u^{m-1/2}\big|^2\,\frac{dx}{|x|^\beta}+(2+\beta-\gamma)^2\int_{\R^d}|x|^{2+\beta-2\gamma}\,u\,dx\\
-\frac2m\,(2+\beta-\gamma)\,(1-m)\,(d-\gamma)\int_{\R^d}u^m\,\frac{dx}{|x|^\gamma}\,.
\end{multline*}
The proportionality constant $\tfrac{1-m}m\,(2+\beta-\gamma)^2$ is such that the coefficient of the moment $\int_{\R^d}|x|^{2+\beta-2\gamma}\,u\,dx$ vanishes. A lengthy but elementary computation shows that $\nrm{\mathfrak B_{\beta,\gamma}^m}{1,\gamma}=\nrm v{2p,d-n}^{2\,p\,\frac{n+2-p(n-2)}{n-p(n-4)}}$.\end{proof}
Notice that the result of Proposition~\ref{Prop:EquivEP-NSII} also holds for a function $u$ such that $\nrm u{1,\gamma}=M\neq M_\star$ if we take the free energy with respect to the Barenblatt profile with same mass.

%%%%%%%%%%%%%%%%%%%%%%%%%%%%%%%%%%%%%%%%%%%%%%%%%%%%%%%%%%%%%%%%%%%%%%
\subsection{Linearization in the entropy -- entropy production framework}\label{Sec:NSIIneq}

A simple computation shows that, in the expression of $\mathcal H[v]$,
\[
\mathrm A=\tfrac{(p-1)^2}{4\,p\,\alpha^2}\,\mathrm C\,,\quad\mathrm B=\tfrac{n-p\,(n-2)}{2\,p}\,\mathrm C\,,\quad\mbox{and}\quad\mathrm C=\tfrac{n+2-p\,(n-2)}{n-p\,(n-4)}\,\nrm{v_\star}{2p,d-n}^{-\,2\,p\,\frac{2\,(p-1)}{n-p(n-4)}}\,.
\]
For a given function $v\in C_0^\infty(\R^d)$, let us consider $v_\scaling(x):=\scaling^\frac n{2p}\,v(\scaling\,x)$ for any $x\in\R^d$. An optimization of
\begin{multline*}
\mathsf h(\scaling):=\mathcal H[v_\scaling]=\tfrac{\mathrm A}2\,\nrm{\D v}{2,d-n}^2\,\scaling^{\frac np-n+2}+\tfrac{\mathrm B}{p+1}\,\nrm v{p+1,d-n}^{p+1}\,\scaling^{n\,\frac{p+1}{2\,p}-n}\\
-\tfrac1{2\,p}\,\nrm v{2p,d-n}^{2\,p\,\frac{n+2-p(n-2)}{n-p(n-4)}}
\end{multline*}
with respect to $\scaling>0$ shows the existence of a unique minimizer $\scaling_\star>0$, for which
\[
2\,p\,\mathsf h(\scaling_\star)=\(\mathsf K_{\alpha,n,p}^\star\,\nrm{\D v}{2,d-n}^\vartheta\,\nrm v{p+1,d-n}^{1-\vartheta}\)^{2\,p\,\frac{n+2-p(n-2)}{n-p(n-4)}}-\nrm v{2p,d-n}^{2\,p\,\frac{n+2-p(n-2)}{n-p(n-4)}}.
\]
In case of symmetry, the inequality $\mathcal H[v]\ge0$ is therefore equivalent to~\eqref{CKN1}. Without symmetry, a similar computation shows that $\mathcal H[v]\ge\inf\mathcal H$ is also equivalent to~\eqref{CKN1}.

{}From the point of view of symmetry breaking, whether the minimum of the functional is achieved by $v_\star$, \emph{i.e.}, $\mathcal H[v]\ge\mathcal H[v_\star]=0$, or if $\inf\mathcal H$ is negative corresponds either to the symmetry case, or to the symmetry breaking case. A Taylor expansion of the functional $\mathcal H$ around $v_\star$ gives rise to the quadratic form defined by
\[
\iRd{|\D g|^2\,|x|^{n-d}}+p\,\frac{\mathrm B}{\mathrm A}\iRd{|g|^2\,\frac{|x|^{n-d}}{1+|x|^2}}-(2\,p-1)\,\frac{\mathrm C}{\mathrm A}\iRd{|g|^2\,\frac{|x|^{n-d}}{(1+|x|^2)^2}}
\]
which, up to a multiplication by a positive constant, coincides with $\mathcal Q[g]$ defined in Section~\ref{Sec:VarCKN-LinearStab}. Hence the discussion of the linear instability will be exactly the same.

%%%%%%%%%%%%%%%%%%%%%%%%%%%%%%%%%%%%%%%%%%%%%%%%%%%%%%%%%%%%%%%%%%%%%%
\subsection{Linearization of the weighted fast diffusion equation}\label{Sec:LinearizationFlow}

Now let us turn our attention to flow issues. The change of variables $v(t,r,\omega)=w(t,r^\alpha,\omega)$ transforms~\eqref{Eqn:FD-FP} into
\be{Eqn:FD-FP2}
w_t-\,\mathsf D_\alpha^*\,\Big[\,w\,\D\(w^{m-1}-|x|^2\)\Big]=0
\ee
upon defining $\mathsf D_\alpha^*$ as the adjoint to $\D$ on $\L^2(\R^d,|x|^{n-d}\,dx)$ so that, if $\mathbf f$ and $g$ are respectively a vector valued function and a scalar valued function, then
\[
\iRd{\mathbf f\cdot\D g\,|x|^{n-d}}=\iRd{(\mathsf D_\alpha^*\,\mathbf f)\,g\,|x|^{n-d}}\,.
\]
In other words, if we take a representation of $\mathbf f$ adapted to spherical coordinates, $s=|x|$ and $\omega=x/s$, and consider $f_s:=\mathbf f\cdot \omega$ and $\mathbf f_\omega:=\mathbf f-f_s\,\omega$, then
\[
\mathsf D_\alpha^*\mathbf f=-\,\alpha\,s^{1-n}\,\frac\partial{\partial s}\(s^{n-1}\,f_s\)-\,\frac1s\,\nabla_{\!\omega}\cdot\mathbf f_\omega\,,
\]
where $\nabla_{\!\omega}$ denotes the gradient with respect to angular derivatives only. We also obtain that

\[
\mathsf D_\alpha^*\,\Big[\,w_1\,\D w_2\Big]=-\,\D w_1\cdot\D w_2+w_1\,\mathsf D_\alpha^*\(\D w_2\)
\]
where, with $s=|x|$,
\[
-\,\mathsf D_\alpha^*\(\D w_2\)=\frac{\alpha^2}{s^{n-1}}\,\frac\partial{\partial s}\(s^{n-1}\,\frac{\partial w_2}{\partial s}\)+\frac1{s^2}\,\Delta_\omega w_2
\]
and $\Delta_\omega$ represents the Laplace-Beltrami operator acting on $\omega\in\S^{d-1}$.

\medskip With the change of variables $s=r^\alpha$, Barenblatt type stationary solutions $\mathfrak B_{\beta,\gamma}$ are transformed into standard Barenblatt profiles
\[
\mathcal B(x)=\mathfrak B_{0,0}(x)=\(1+|x|^2\)^\frac1{m-1}\quad\forall\,x\in\R^d\,.
\]
The Barenblatt function $\mathcal B$ is expected to attract the solution $w$ to~\eqref{Eqn:FD-FP2}, so that $w/\mathcal B$ converges to $1$ as $t\to+\infty$. We shall prove in~\cite{BDMN2016b} that this holds true in the norm of uniform convergence. This suggests to write $w=\mathcal B\,(1+\varepsilon\,\mathcal B^{1-m}\,\fg)$ as in~\cite{BBDGV} and write a linearized equation for $f$ by formally taking the limit as $\varepsilon\to0$, in order to explore the asymptotic regime as $t\to+\infty$. Hence we obtain the linearized flow
\[\label{Eqn:rellinFD-FP2}
\fg_t-(m-1)\,\mathcal L\,\fg=0
\]
at lowest order with respect to $\varepsilon$, with an operator $\mathcal L$ on $\L^2(\R^d,|x|^{n-d}\,\mathcal B^{2-m}\,dx)$ defined by
\[
\mathcal L\,\fg:=\mathcal B^{m-2}\,\mathsf D_\alpha^*\Big(\,\mathcal B\,\D\fg\Big)\,.
\]

An expansion of $\mathcal F$ and $\mathcal I$ in terms of $v=\mathcal B\(1+\varepsilon\,\fg\,\mathcal B^{1-m}\)$ at order two in $\varepsilon$ gives rise to the expressions
\[
\varepsilon^{-2}\,\mathcal F[v]\sim\tfrac m{2\,\alpha}\,\mathsf F[\fg]\,,\quad\varepsilon^{-2}\,\mathcal I[v]\sim\tfrac{(1-m)^2}\alpha\,\mathsf I[\fg]\,,
\]
where
\[
\mathsf F[\fg]= \int_{\R^d}|\fg|^2\,\mathcal B^{2-m}\,|x|^{n-d}\,dx\quad\mbox{and}\quad\mathsf I[\fg]=\int_{\R^d}|\D\fg|^2\,\mathcal B\,|x|^{n-d}\,dx
\]
while the condition $\int_{\R^d}\fg\,\mathcal B^{2-m}\,|x|^{n-d}\,dx=0$ is satisfied because of the mass conservation. Differentiating $\mathsf F[\fg(t,\cdot)]$ along the linearized flow gives
\[
\frac d{dt}\mathsf F[\fg(t,\cdot)] = -\,2\,(1-m)\,\mathsf I[\fg(t,\cdot)]\,.
\]
The expansion of $\mathcal I-\tfrac{1-m}m\,(2+\beta-\gamma)^2\,\mathcal F$ around $\mathcal B$ also gives, at order $\varepsilon^2$ and with the notations of Section~\ref{Sec:VarCKN-LinearStab},
\[
\tfrac{2\,(1-m)^2}{\mathsf a\,\alpha}\,\mathcal Q[\fg]=\tfrac{(1-m)^2}\alpha\,\mathsf I[\fg]-\tfrac{1-m}m\,(2+\beta-\gamma)^2\,\tfrac m{2\,\alpha}\,\mathsf F[\fg]=\tfrac{(1-m)^2}\alpha\,\Big(\mathsf I[\fg]-\Lambda_\star\,\mathsf F[\fg]\Big)
\]
with
\be{LambdaStar}
\Lambda_\star:=\frac{(2+\beta-\gamma)^2}{2\,(1-m)}=2\,\alpha^2\,\delta
\ee
and the linear instability, that is, the fact that $\mathcal Q$ takes negative values, immediately follows if $\Lambda<\Lambda_\star$. If symmetry holds, we deduce from the \emph{entropy -- entropy production} inequality~\eqref{Ineq:E-EP} that
\[
\mathsf I[\fg]\ge\Lambda_\star\,\mathsf F[\fg]\,.
\]
Even without symmetry, a spectral gap inequality stating that $\mathsf I[\fg]\ge\Lambda\,\mathsf F[\fg]$ with optimal constant $\Lambda$ results in the decay estimate
\[
\mathsf F[\fg(t,\cdot)]\le\mathsf F_0\,e^{-\,2\,(1-m)\,\Lambda\,t}\quad\mbox{as}\quad t\to+\infty\,.
\]
The existence of such a spectral gap inequality is the subject of the next section. This concludes the strategy for the proof of Theorem~\ref{Thm:Asymptotic rates}. Of course, in order to justify this formal approach, one has to establish additional properties, like the \emph{uniform relative convergence}, which means that $v/\mathfrak B$ uniformly converges to $1$: this is the main result of~\cite{BDMN2016b}. Notice that in the non-weighted case $(\beta,\gamma)=(0,0)$ (see~\cite[Corollary~1, page~709]{1004}), the spectral gap constant is given by $\Lambda=2/(1-m)$ and we recover a decay of order $e^{-4t}$ as in~\cite{BBDGV}. As for symmetry breaking, what matters is to compare $\Lambda$ and $\Lambda_\star$: if $\Lambda<\Lambda_\star$, by considering an eigenfunction associated with the spectral gap, we shall prove that $\mathcal Q$ takes negative eigenvalues, and this is what establishes the result of Theorem~\ref{Thm:SymmetryBreaking}. We shall come back to these issues in Sections~\ref{Sec:LinearizationSetting} and~\ref{Sec:SBVar}.

If we denote by $\scal\cdot\cdot$ the natural scalar product on $\L^2(\R^d,|x|^{n-d}\,d\mu_{\delta+1})$ given by $\scal{\fg_1}{\fg_2}=\int_{\R^d}\fg_1\,\fg_2\,|x|^{n-d}\,d\mu_{\delta+1}$ where $\delta=1/(1-m)$ and $\mu_\delta(x)=(1+|x|^2)^{-\delta}$ as in Proposition~\ref{Prop:SpectralGap}, then the linearized free energy and the linearized Fisher information take the form
\[
\mathsf F[\fg]=\scal\fg\fg\quad\mbox{and}\quad\mathsf I[\fg]=\scal\fg{\mathcal L\,\fg}\,,
\]
and $\mathcal L$ is self-adjoint on $\L^2(\R^d,|x|^{n-d}\,d\mu_{\delta+1})$. We are now ready to study the spectrum of $\mathcal L$.

%%%%%%%%%%%%%%%%%%%%%%%%%%%%%%%%%%%%%%%%%%%%%%%%%%%%%%%%%%%%%%%%%%%%%%
%%%%%%%%%%%%%%%%%%%%%%%%%%%%%%%%%%%%%%%%%%%%%%%%%%%%%%%%%%%%%%%%%%%%%%
\section{Spectral properties of the linearized operator and consequences}\label{Sec:SpectrumResults}

%%%%%%%%%%%%%%%%%%%%%%%%%%%%%%%%%%%%%%%%%%%%%%%%%%%%%%%%%%%%%%%%%%%%%%
\subsection{Results on the spectrum}\label{Sec:Spectrum}

Non-constant coefficients of $\mathcal L$ are invariant under rotations with respect to the origin, so that a spherical harmonics decomposition can be made to compute the spectrum. Let $\mu_\kell=\kell\,(\kell+d-2)$, $\kell\in\N$ be the sequence of the eigenvalues of the Laplace-Beltrami operator on $\S^{d-1}$. The problem is reduced to find the critical values $\Lambda_{\ellk,\kell}\in[0,\Lambda_{\rm ess})$ of the Rayleigh quotient
\[\label{RQ1}
f\mapsto\frac{\int_0^\infty\(\alpha^2\,|f'(s)|^2+\frac{\mu_\kell}{s^2}\,|f(s)|^2\)\frac{s^{n-1}\,ds}{(1+s^2)^\delta}}{\int_0^\infty|f(s)|^2\,\frac{s^{n-1}\,ds}{(1+s^2)^{\delta+1}}}\,.
\]
Here we take the convention to index $\Lambda_{\ellk,\kell}$ with $\ellk$, $\kell\in\N$. The spectral component $\kell=0$ corresponds to radial functions in $\L^2(\R^d,d\mu_{\delta+1})$. Alternatively, the problem is reduced to find the eigenvalues $\Lambda_{\ellk,\kell}$ defined by the Euler-Lagrange equations associated to the Rayleigh quotient, \emph{i.e.},
\be{EV}
-\,\alpha^2\,\frac d{ds}\left[\frac{s^{n-1}}{(1+s^2)^\delta}\,f_{\ellk,\kell}'(s)\right]+\frac{\mu_\kell\,s^{n-3}}{(1+s^2)^\delta}\,f_{\ellk,\kell}(s)-\frac{\Lambda_{\ellk,\kell}\,s^{n-1}}{(1+s^2)^{\delta+1}}\,f_{\ellk,\kell}(s)=0\,.
\ee
Our key technical result is the following lemma. Complements can be found in Appendix~\ref{Sec:SpectrumAditional}.
%---------------------------------------------------------------------
\begin{lemma}\label{Lem:SpectrumResults} Let $d\ge2$, $\alpha\in(0,+\infty)$, $n>d$ and $\delta>0$. Then the following properties hold:
\begin{enumerate}
\item[(i)] The kernel of $\mathcal L$ on the space $\L^2(\R^d,|x|^{n-d}\,d\mu_{\delta+1})$ is generated by the constants. As a consequence,
\[
\Lambda_{0,0}=0\,.
\]
\item[(ii)] The essential spectrum of the operator $\mathcal L$ is the interval $[\Lambda_{\rm ess},\infty)$ with
\[
\Lambda_{\rm ess}=\tfrac14\,\alpha^2\,\(n-2-2\,\delta\)^2\,.
\]
\item[(iii)] In the radial component, the first positive eigenvalue of $\mathcal L$ is given by
\[
\Lambda_{1,0}=2\,\alpha^2\,(2\,\delta-n)
\]
in the range $\delta>1+\frac n2$ and the corresponding eigenspace contains
\[
f_{1,0}(s)=s^2-\frac n{2\,\delta-n}\quad\forall\,s\ge0\,.
\]
There is no such eigenvalue if $0<\delta\le1+\frac n2$.
\item[(iv)] The smallest eigenvalue of $\mathcal L$ corresponding to a non-radial component is
\[
\Lambda_{0,1}=\,2\,\alpha^2\,\delta\,\eta\,,
\]
in the range $\delta>\eta+\frac{n-2}2+\frac{\sqrt{d-1}}\alpha$, where $\eta$ is the unique positive solution to~\eqref{Eqn:eq-eta}. The corresponding eigenspace is one-dimensional and generated by
\[
f_{0,1}(s)=s^\eta\quad\forall\,s\ge0\,.
\]
\end{enumerate}
\end{lemma}
%---------------------------------------------------------------------
The results of the Lemma~\ref{Lem:SpectrumResults} are illustrated in Figs.~\ref{Fig:F2} and ~\ref{Fig:F3}. Also see Appendix~\ref{Sec:SpectrumAditional} for further details on the spectrum of $\mathcal L$. Notice that the whole range $\delta>0$ is covered, while results deduced from~\eqref{CKN} require $\delta\ge n$.
\begin{proof}Each of the properties relies on elementary considerations.
\\
(i) The kernel of $\mathcal L$ is characterized by the equation $\D f=0$.
\\
(ii) According to Persson's lemma in~\cite{MR0133586}, the infimum of the essential spectrum of the operator $\mathcal L$ is given by the Hardy inequality, as in~\cite{MR2320246},
\begin{multline*}
\iRd{|\D\fg|^2\,|x|^{n-d}}-\(\Lambda_{\rm ess}+\alpha^2\,\delta\,(n-2-\delta)\)\iRd{|\fg|^2\,|x|^{n-d-2}}\\
\ge\(\tfrac14\,\alpha^2\,(n-2)^2-\Lambda_{\rm ess}-\alpha^2\,\delta\,(n-2-\delta)\)\iRd{|\fg|^2\,|x|^{n-d-2}}
\end{multline*}
if we request that the coefficient of the right hand side is actually zero. Hence,
\[
\Lambda_{\rm ess}:=\tfrac14\,\alpha^2\,(n-2)^2-\alpha^2\,\delta\,(n-2-\delta)=\tfrac14\,\alpha^2\,\(n-2-2\,\delta\)^2\,.
\]
\\
(iii) By direct computation, we find that $f_{1,0}(s)=s^2-\frac n{n-2\,\delta}$ and $\Lambda_{1,0}=2\,\alpha^2\,(\delta-\,2\,n)$ solve~\eqref{EV}. Notice that this mode does not break the symmetry since it corresponds to a radial mode ($k=0$). It is orthogonal to $f_{0,0}=1$ and corresponds to $\ellk=1$ by the Sturm-Liouville theory.
\\
(iv) We can check that $f_{0,1}(s)=s^\eta$ and $\Lambda_{0,1}=\,2\,\alpha^2\,\delta\,\eta$ provides a solution to~\eqref{EV}, hence the ground state in the $\kell=1$ component because $f_{0,1}$ is nonnegative, as soon as $\eta>0$ solves~\eqref{Eqn:eq-eta}. Up to a multiplication by a constant, it is unique by the Sturm-Liouville theory: another independent eigenfunction would have to change sign and its positive part would also be a solution with support strictly included in $\R^+$, a contradiction with the unique continuation property of the solution to the ordinary differential equation.

The admissibility of the functions $f_{1,0}$ and $f_{0,1}$ as well as the ranges of existence of the lowest eigenvalues in terms of~$\delta$ are discussed in Appendix~\ref{Sec:SpectrumAditional}. Also see Fig.~\ref{Fig:F4}.\end{proof}

We are now able to prove Proposition~\ref{Prop:SpectralGap}. We recall that $\Lambda_{1,0}=\Lambda_{0,1}$ determines the curve $\beta=\sigma(\gamma,p)$: see Figs.~\ref{Fig:F2},~\ref{Fig:F3},~\ref{Fig:F4} and~\ref{Fig:F5}.

\begin{proof}[Proof of Proposition~\ref{Prop:SpectralGap}] Notice first that~\eqref{Eqn:SpectralGap} only expresses that
\[
\Lambda=\min\{\Lambda_{1,0},\Lambda_{0,1}\}\,.
\]
Indeed, we have that
\[
\Lambda_{0,1}-\Lambda_{1,0}=2\,\alpha^2\,\delta\,\left(\eta-2+\frac n\delta\right)
\]
is negative if and only if
\[
\eta<2-\frac n\delta:=\eta_0\,.
\]
Since the map $x\mapsto x\,(x+n-2)$ is increasing on $\R_+$, the condition $\eta<\eta_0$ is equivalent to
\[
\alpha^2>\frac{d-1}{\eta_0\,(\eta_0+n-2)}=\frac{(d-1)\,\delta^2}{n\,(2\,\delta-n)\,(\delta-1)},
\]
and~\eqref{Eqn:SpectralGap} directly follows.

It remains to prove that $\Lambda_{\rm ess}>\min\{\Lambda_{1,0},\Lambda_{0,1}\}$. Let us first compute
\[
\Lambda_{\rm ess}-\Lambda_{1,0}=\frac14\,\alpha^2\,(n-2-2\,\delta)^2-2\,\alpha^2\,(2\,\delta-n)=\frac14\,\alpha^2\,(2\,\delta-n-2)^2
\]
and observe that it is positive since $\delta=(n+2)/2>n$ if and only if $n<2$. As a consequence, we obtain that
\[
\Lambda=\min\{\Lambda_{\rm ess},\Lambda_{1,0},\Lambda_{0,1}\}=\min\{\Lambda_{1,0},\Lambda_{0,1}\}
\]
in the range defined by~Ê\eqref{parameters}. This completes the proof.
\end{proof}

%---------------------------------------------------------------------
\begin{corollary}\label{Cor:SB} Under the assumptions of Lemma~\ref{Lem:SpectrumResults}, $\Lambda_\star$ is larger than $\Lambda$ if and only if~\eqref{FS} holds.\end{corollary}
%---------------------------------------------------------------------
\begin{proof} We recall that by~\eqref{LambdaStar}, $\Lambda_\star=2\,\alpha^2\,\delta$. An elementary computation gives
\[
\Lambda_{1,0}-\Lambda_\star=2\,\alpha^2\,(\delta-n)\quad\mbox{and}\quad\Lambda_{0,1}-\Lambda_\star=2\,\alpha^2\,\delta\,\big(\eta-1\big)\,,
\]
hence $\Lambda_{0,1}<\Lambda_\star$ if and only if $\eta<1$, that is, $\alpha^2>\frac{d-1}{n-1}$ which amounts to~\eqref{FS} and completes the proof.\end{proof}
%---------------------------------------------------------------------
\begin{figure}[ht]
\begin{center}\hspace*{-1cm}
\includegraphics[width=6.5cm]{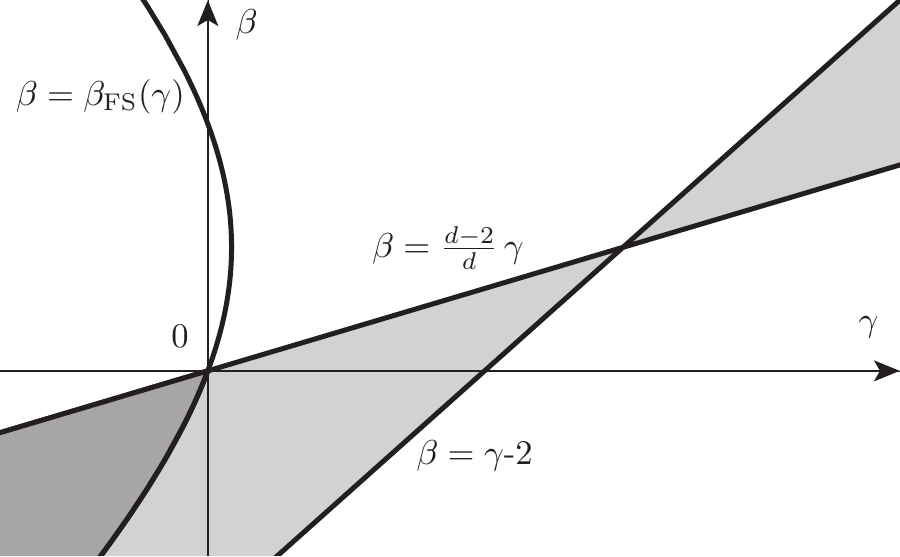}\hspace*{6pt}\includegraphics[width=6.5cm]{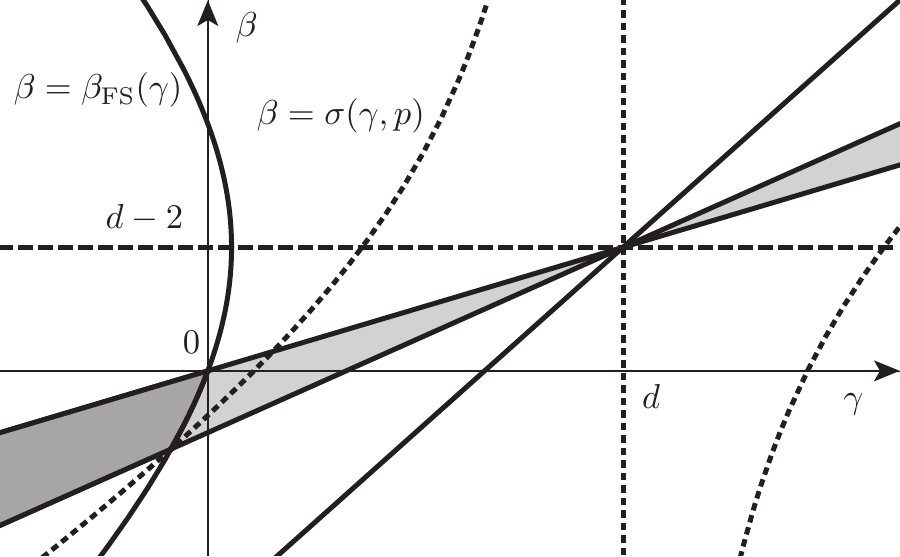}
\caption{\label{Fig:F2} With $d=3$, the left figure is essentially an enlargement of Fig.~\ref{Fig:F1} and represents the symmetry breaking region, while on the right figure, we choose $p=2$, so that the admissible range of parameters $(\beta,\gamma)$ is restricted by the condition $p\le p_\star(\beta,\gamma)$, \emph{i.e.}, $\beta\ge d-2-(d-\gamma)/p$. This lower bound corresponds to the line determined by the points $(\beta,\gamma)=(d-2,d)$ and $(\beta,\gamma)$ given by the condition $\Lambda_\star=\Lambda_{0,1}=\Lambda_{1,0}$. The curve $\beta=\sigma(\gamma,p)$ in Theorem~\ref{Thm:Asymptotic rates} is represented by a dotted curve. To $\beta\ge\sigma(\gamma,p)$ corresponds the case $\Lambda=\Lambda_{0,1}\le\Lambda_{1,0}$, while $\beta\le\sigma(\gamma,p)$ corresponds to the case $\Lambda_{0,1}\ge\Lambda_{1,0}=\Lambda$, when $\gamma\in(-\infty,d)$.}
\end{center}
\end{figure}
%---------------------------------------------------------------------
\begin{figure}[ht]
\begin{center}\hspace*{-1cm}
\includegraphics[width=6.5cm]{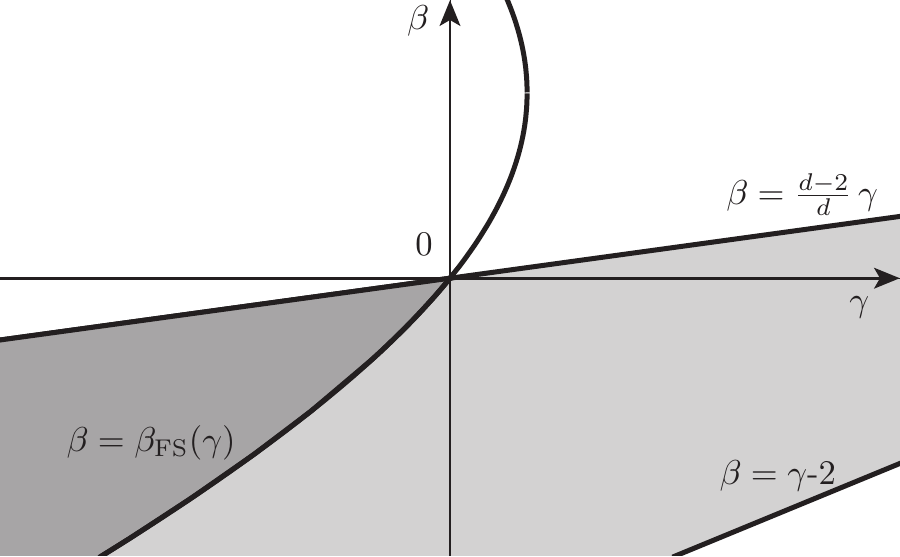}\hspace*{6pt}\includegraphics[width=6.5cm]{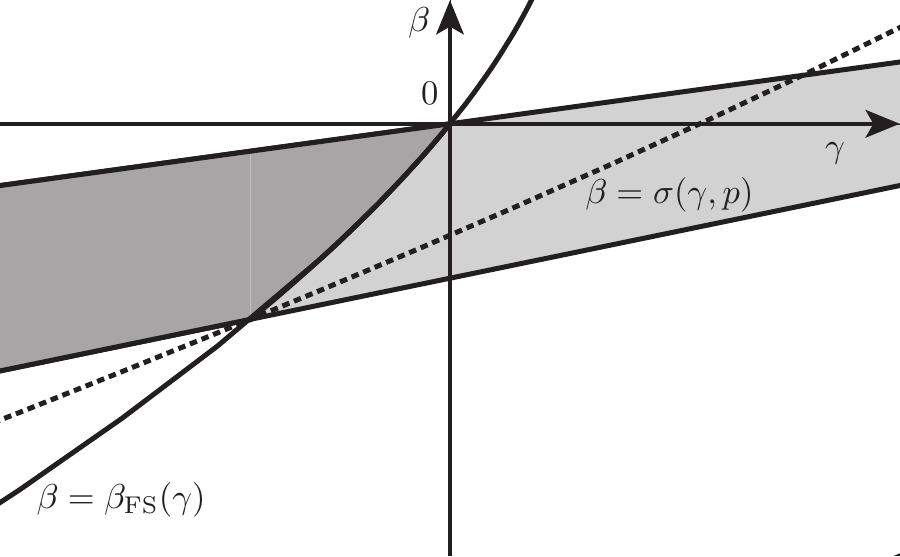}
\caption{\label{Fig:F3} Enlargement of Fig.~\ref{Fig:F2} in a neighborhood of $(\beta,\gamma)=(0,0)$. On the right, the equality case $\Lambda_{0,1}=\Lambda_{1,0}$ determines the dotted curve $\beta=\sigma(\gamma,p)$. Notice that the symmetry breaking region is contained in the region in which the spectral gap is $\Lambda=\Lambda_{0,1}$.}
\end{center}
\end{figure}
%---------------------------------------------------------------------

%%%%%%%%%%%%%%%%%%%%%%%%%%%%%%%%%%%%%%%%%%%%%%%%%%%%%%%%%%%%%%%%%%%%%%
\subsection{Symmetry breaking: a proof based on the nonlinear flow}\label{Sec:LinearizationSetting}

Here we admit the result of Theorem~\ref{Thm:Asymptotic rates}, which is proved in~\cite{BDMN2016b}.

\begin{proof}[Proof of Theorem~\ref{Thm:SymmetryBreaking}] If symmetry holds~\eqref{CKN}, then the \emph{entropy -- entropy production} inequality~\eqref{Ineq:E-EP} also holds and it is then clear by considering the large time asymptotics of the solution to~\eqref{Eqn:FD-FP} that the estimate $\mathcal F[v(t,\cdot)]\le O\(e^{-\,2\,(1-m)\,\Lambda_\star\,t}\)$ given by~\eqref{EntropyDecay} and~\eqref{LambdaStar} is not compatible with $\mathcal F[v(t,\cdot)]=O\(e^{-\,2\,(1-m)\,\Lambda_{0,1}\,t}\)$ if
\be{SBCdt}
4\,\alpha^2=(2+\beta-\gamma)^2=2\,(1-m)\,\Lambda_\star>2\,(1-m)\,\Lambda_{0,1}=\,4\,\alpha^2\,\eta\,.
\ee
Indeed, we may use an eigenfunction $f_{0,1}$ associated with $\Lambda_{0,1}$ to consider a perturbation of the Barenblatt function, that is, we can test the quotient $\mathcal I/\mathcal F$ by $\mathfrak B\(1+\varepsilon\,f_{0,1}\,\mathfrak B^{1-m}\)$ and let $\varepsilon\to0$. Hence, if~\eqref{FS} holds, \emph{i.e.}, if $d-1-(n-1)\,\alpha^2<0$, then $\eta$ given by~\eqref{Eqn:eq-eta} satifies $\eta<1$ and symmetry breaking occurs. It is shown in Section~\ref{Sec:CKNrange} that this provides us with the condition $\beta>\beta_{\rm FS}(\gamma)$ in Theorem~\ref{Thm:SymmetryBreaking}.\end{proof}

%%%%%%%%%%%%%%%%%%%%%%%%%%%%%%%%%%%%%%%%%%%%%%%%%%%%%%%%%%%%%%%%%%%%%%
\subsection{Symmetry breaking: a variational approach}\label{Sec:SBVar}

To make this paper self-contained, we give a variational proof of Theorem~\ref{Thm:SymmetryBreaking} based on the more standard, variational approach of~\cite{Catrina-Wang-01,Felli-Schneider-03} for~\eqref{CKN-DEL}.

\medskip As a corollary of Lemma~\ref{Lem:SpectrumResults}, we determine the optimal constant $\kappa$ in
\begin{multline}
\iRd{|\D g|^2\,|x|^{n-d}}+\kappa\iRd{|g|^2\,\frac{|x|^{n-d}}{1+|x|^2}}\\
\ge p\,(2\,p-1)\,\frac{(2+\beta-\gamma)^2}{(p-1)^2}\iRd{|g|^2\,\frac{|x|^{n-d}}{(1+|x|^2)^2}}\,.\label{Hardy-Poincare2}
\end{multline}
%---------------------------------------------------------------------
\begin{corollary}\label{Cor:SpectrumEquivalence} If $\delta=\frac{2\,p}{p-1}$, Inequalities~\eqref{Hardy-Poincare1} and~\eqref{Hardy-Poincare2} are equivalent and their optimal constants are related by
\[
\kappa=\frac p{(p-1)^2}\,(2+\beta-\gamma)\,\big[d-2-\beta-p\,(d+\gamma-2\,\beta-4)\big]-\Lambda\,.
\]
Moreover the two operators associated with the quadratic forms have the same spectral gaps.\end{corollary}
%---------------------------------------------------------------------
\begin{proof} If we define $\fg=(1+|x|^2)^{\delta/2}\,g$, then the spectral gap problem~\eqref{Hardy-Poincare1} is reduced to the equivalent problem of finding the largest positive $\Lambda$ such that the quadratic form
\[
\iRd{|\D g|^2\,|x|^{n-d}}+\iRd{\(\frac{\alpha^2\,\delta\,(\delta+2-n)-\Lambda}{1+|x|^2}-\frac{\alpha^2\,\delta\,(\delta+2)}{(1+|x|^2)^2}\)|g|^2\,|x|^{n-d}}
\]
is nonnegative. We conclude by identifying the terms with those in~\eqref{Hardy-Poincare2} and replacing $\alpha$ and $\delta$ by their values in terms of $\beta$ and $\gamma$.\end{proof}

Let us consider the functional $\mathcal Q$ as defined in Section~\ref{Sec:VarCKN-LinearStab}. Using
\[
g_{0,1}(x):=(1+|x|^2)^{-\delta/2}\,\fg_{0,1}(x)
\]
as a test function, where $\fg_{0,1}$ is an eigenfunction associated with $\Lambda_{0,1}$, we observe that $\mathcal Q[g_{0,1}]<0$ if and only~if
\begin{multline*}
\frac{p\,(2+\beta-\gamma)}{(p-1)^2}\,\big[d-\gamma-p\,(d-2-\beta)\big]\\
<\kappa:=\frac p{(p-1)^2}\,(2+\beta-\gamma)\,\big[d-2-\beta-p\,(d+\gamma-2\,\beta-4)\big]-\Lambda_{0,1}
\end{multline*}
where the right-hand side follows from Corollary~\ref{Cor:SpectrumEquivalence} when $\Lambda=\Lambda_{0,1}$, \emph{i.e.}, when $\beta\ge\sigma(\gamma,p)$. In any case, we find that $\mathcal Q[g_{0,1}]<0$ if $\Lambda_{0,1}<2\,\alpha^2\,\delta$. Hence we recover the condition of Theorem~\ref{Thm:SymmetryBreaking} as in Section~\ref{Sec:LinearizationSetting}.

%%%%%%%%%%%%%%%%%%%%%%%%%%%%%%%%%%%%%%%%%%%%%%%%%%%%%%%%%%%%%%%%%%%%%%
%%%%%%%%%%%%%%%%%%%%%%%%%%%%%%%%%%%%%%%%%%%%%%%%%%%%%%%%%%%%%%%%%%%%%%
\section{Conclusions}\label{Sec:Conclusion}

Let us summarize what we have learned in this paper so far. Three interpolation inequalities have been considered:
\begin{itemize}
\item the Caffarelli-Kohn-Nirenberg inequalities~\eqref{CKN},
\item the entropy -- entropy production inequality~\eqref{Ineq:E-EP},
\item the weighted Gagliardo-Nirenberg inequality~\eqref{CKN1}.
\end{itemize}
In case of symmetry, these three inequalities are equivalent and the linear stability of the radial optimal functions has been reduced to the discussion of the sign of the quadratic form $\mathcal Q$, that is, of the sign of $\Lambda_\star-\Lambda_{0,1}$: according to~\eqref{SBCdt}, whenever it is positive, we know that symmetry breaking occurs. As observed in Section~\ref{Sec:LinearizationSetting}, this is consistent with the dynamic point of view. When symmetry occurs, the global rate of convergence of the entropy is bounded by $\Lambda_\star$ and the slowest asymptotic rate of convergence is determined by $\Lambda_{0,1}$, so that $\Lambda_\star-\Lambda_{0,1}$ has to be nonpositive: if $\Lambda_\star-\Lambda_{0,1}>0$, then by contradiction symmetry breaking occurs.

The spectral gap in the Hardy-Poincar\'e inequality~\eqref{Hardy-Poincare1} determines the worst asymptotic rate of convergence of a solution, and this rate is sharp. We observe three possible regions, which are shown in Fig.~\ref{Fig:F4}.
\begin{itemize}
\item[]\hspace*{-12pt} Region \circled{1} : $\Lambda=\Lambda_{0,1}<\Lambda_\star<\Lambda_{1,0}<\Lambda_{\rm ess}$, symmetry breaking occurs,
\item[]\hspace*{-12pt} Region \circled{2} : $\Lambda_\star<\Lambda=\Lambda_{0,1}<\Lambda_{1,0}<\Lambda_{\rm ess}$,
\item[]\hspace*{-12pt} Region \circled{3} : $\Lambda_\star<\Lambda=\Lambda_{1,0}<\Lambda_{0,1}<\Lambda_{\rm ess}$.
\end{itemize}
Of course, one can consider the threshold cases in which some inequalities become equalities. For instance in the limit case $(\beta,\gamma)=(0,0)$, it turns out that $\Lambda_\star=\Lambda_{0,1}$. We also have to notice that $\Lambda_\star-\Lambda_{0,1}>0$ is only a sufficient condition for symmetry breaking, for which we know that $\C_{\beta,\gamma,p}>\C_{\beta,\gamma,p}^\star$, but the actual region for symmetry breaking could \emph{a priori} be larger than \circled{1}. Actually, based on recent results obtained in~\cite{DELM2015}, we learn that symmetry holds in \circled{2} and \circled{3}.
%---------------------------------------------------------------------
\begin{figure}[ht]
\begin{center}\hspace*{-1cm}
\includegraphics[width=10cm]{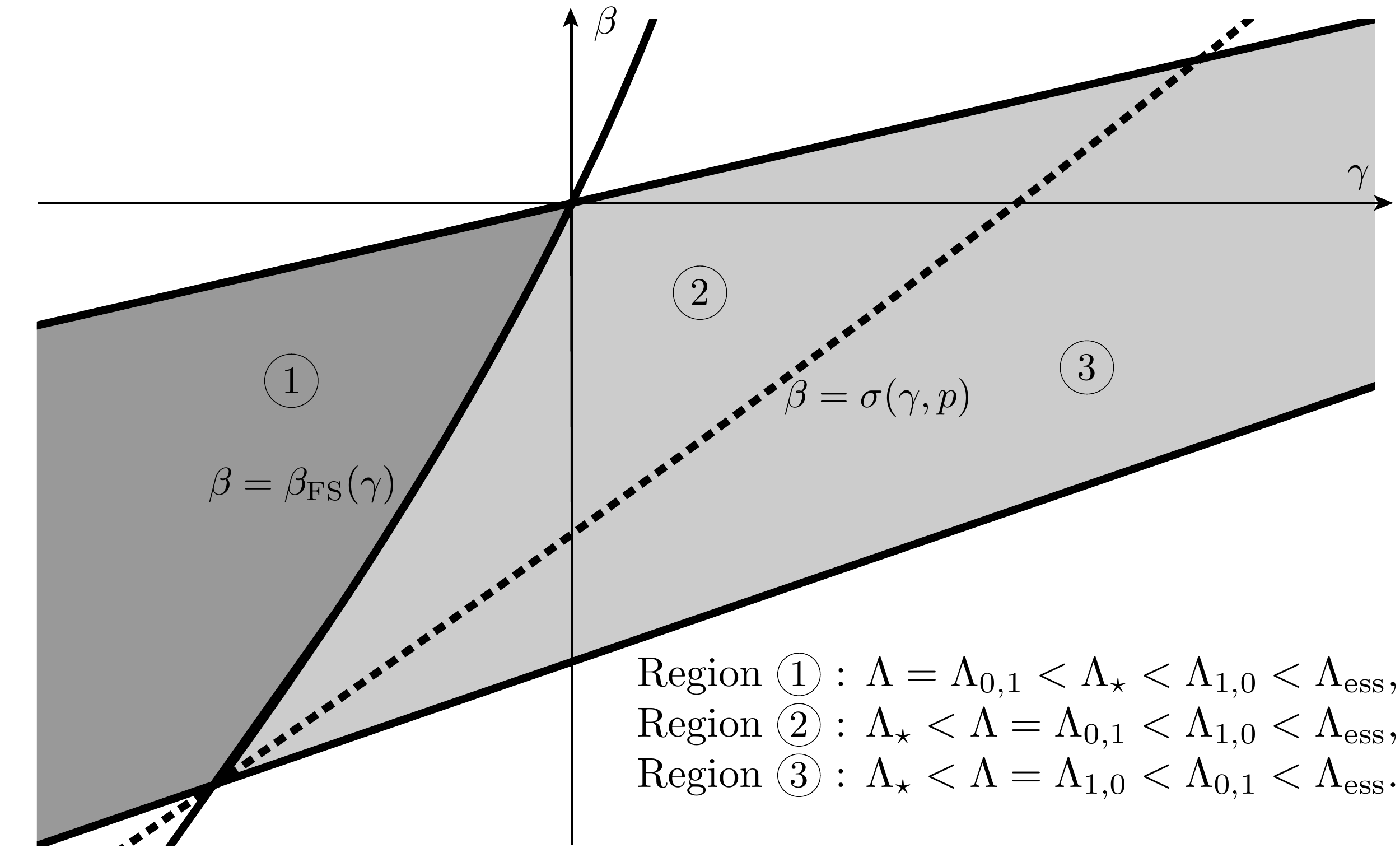}
\caption{\label{Fig:F4} In the dark grey region, symmetry breaking occurs. The plot is done for $p=2$ and $d=3$. See Appendix~\ref{Sec:SpectrumAditional} for a more detailed description of the properties of the lowest eigenvalues.}
\end{center}
\end{figure}
%---------------------------------------------------------------------

\medskip In Region \circled{1}, we know from Section~\ref{Sec:CKNrange} that
\[
\C_{\beta,\gamma,p}^\star<\C_{\beta,\gamma,p}<\mathcal C_{a,b}^\vartheta
\]
with $\vartheta$ as in~\eqref{theta}. The value of $\mathcal C_{a,b}$ is not explicitly known because symmetry breaking occurs for the corresponding values of $a=\beta/2$ and $b=\gamma/(2\,p_\star)$ according to~\cite{Felli-Schneider-03}, but at least $\mathcal C_{a,b}$ can be estimated in terms of $\mathcal C_{a,a}$ and $\mathcal C_{a,a+1}$, which are both explicitly known: see for instance~\cite{Catrina-Wang-01}.

Now let us turn our attention to the entropy -- entropy production inequality
\be{EP}
\mathcal K(M)\,\mathcal F[v]\le\mathcal I[v]\quad\forall\,v\in\L^{1,\gamma}(\R^d)\quad\mbox{such that}\quad \nrm v{1,\gamma}=M\,,
\ee
where $\mathcal K(M)$ is the best constant. The first question to decide is whether such an inequality makes sense for some $\mathcal K(M)>0$. If symmetry holds, for instance if $\beta=0$ and $\gamma>0$ is small according to~\cite{DMN2015}, we already know that the answer is yes and that $\mathcal K(M)\ge\tfrac{1-m}m\,(2+\beta-\gamma)^2$ because of~\eqref{Ineq:E-EP}. A more complete answer is given by the following result
%---------------------------------------------------------------------
\begin{proposition}\label{Prop:E-EP} With the notations of Theorem~\ref{Thm:SymmetryBreaking} and under the symmetry breaking assumption
\[
\gamma<0\quad\mbox{and}\quad\beta_{\rm FS}(\gamma)<\beta<\frac{d-2}d\,\gamma\,,
\]
for any $M>0$, we have $0<\mathcal K(M)\le\frac2m\,(1-m)^2\,\Lambda_{0,1}$. On the other hand, under the condition
\[
0<\gamma\le d\,,\quad\mbox{or}\quad\gamma\le0\quad\mbox{and}\quad\gamma-2<\beta<\beta_{\rm FS}(\gamma)\,,
\]
we have $\mathcal K(M)>\tfrac{1-m}m\,(2+\beta-\gamma)^2$.\end{proposition}
%---------------------------------------------------------------------
\begin{proof} Let us consider a minimizing sequence $(v_n)_{n\in\N}$ of $v\mapsto\mathcal I[v]/\mathcal F[v]$, such that $\nrm{v_n}{1,\gamma}=M$ and $\mathcal F[v_n]>0$ for any $n\in\N$. In region \circled{1}, we know that $\lim_{n\to\infty}\mathcal I[v_n]/\mathcal F[v_n]=\mathcal K(M)\le\frac2m\,(1-m)^2\,\Lambda_{0,1}<\frac2m\,(1-m)^2\,\Lambda_\star$. Hence we deduce from the fact that
\[
\mathcal I[v_n]-\frac2m\,(1-m)^2\(\Lambda_\star-\varepsilon\)\mathcal F[v_n]\le0
\]
for any $n$ large enough and for some $\varepsilon>0$ small enough that $\nrm{\nabla v_n^{m-1/2}}{2,\beta}$, $\nrm{v_n^m}{2,\gamma}$ and $\iRd{|x|^{2+\beta-2\gamma}\,v_n}$ are all bounded uniformly in $n$. We may pass to the limit as $n\to\infty$ and get that $\mathcal K(M)>0$ is achieved by some function $v\not\equiv v_\star$ with $\nrm v{1,\gamma}=M$.

By arguing as in~\cite{dolbeault:hal-01081098} in regions \circled{2} and \circled{3}, one can obtain an improved version of the inequality which shows that, if $\mathcal K(M)=\tfrac{1-m}m\,(2+\beta-\gamma)^2$, any minimizing sequence $(v_n)_{n\in\N}$ has to converge to $v_\star$, up to a scaling. But then we obtain $\mathcal K(M)=\lim_{n\to\infty}\mathcal I[v_n]/\mathcal F[v_n]\ge\frac2m\,(1-m)^2\,\Lambda>\frac2m\,(1-m)^2\,\Lambda_\star=\mathcal K(M)$, a contradiction.\end{proof}

{}In the symmetry range, that is, if either $0\le\gamma\le d$, or $\gamma<0$ and $\gamma-2<\beta\le\beta_{\rm FS}(\gamma)$, we know  that there is an entropy -- entropy production inequality~\eqref{EP} for some $\mathcal K(M)>0$. In that range, then $\mathcal K(M)=\tfrac{1-m}m\,(2+\beta-\gamma)^2$ if and only if $\beta=\beta_{\rm FS}(\gamma)$. In terms of the evolution equation~\eqref{FD}, we conclude that there is a global exponential rate of convergence of the free energy, \emph{i.e.},
\[
\mathcal F[v(t,\cdot)]\le\mathcal F[u_0]\,e^{-\,\frac m{1-m}\,\mathcal K(M)\,t}\quad\forall\,t\ge0\,,
\]
whenever~\eqref{EP} holds, but this global rate is given by $\mathcal K(M)=\frac{1-m}m\,(2+\beta-\gamma)^2$ if and only if $\beta=\beta_{\rm FS}$. Moreover, when $(\beta,\gamma)=(0,0)$, $\frac m{1-m}\,\mathcal K(M)=2\,(1-m)\,\Lambda_\star$ with $\Lambda_\star=\Lambda_{0,1}$, so that the global rate is the same as the asymptotic one obtained by linearization, and the corresponding eigenspace can be identified by considering the translations of the Barenblatt profiles. For further considerations on the case $(\beta,\gamma)\neq(0,0)$, see Appendix~\ref{Sec:SpectrumAditional}.

\medskip For all consequences for the nonlinear evolution equation~\eqref{FD} and their proofs, the reader is invited to refer to the second part of this work: see~\cite{BDMN2016b}. Altogether Proposition~\ref{Prop:E-EP} shows that a fast diffusion equation with weights like~\eqref{FD} has properties which definitely differ from similar equations without weights.

%%%%%%%%%%%%%%%%%%%%%%%%%%%%%%%%%%%%%%%%%%%%%%%%%%%%%%%%%%%%%%%%%%%%%%
%%%%%%%%%%%%%%%%%%%%%%%%%%%%%%%%%%%%%%%%%%%%%%%%%%%%%%%%%%%%%%%%%%%%%%
\bigskip\begin{center}\rule{4cm}{0.5pt}\end{center}\medskip\appendix
%%%%%%%%%%%%%%%%%%%%%%%%%%%%%%%%%%%%%%%%%%%%%%%%%%%%%%%%%%%%%%%%%%%%%%

%%%%%%%%%%%%%%%%%%%%%%%%%%%%%%%%%%%%%%%%%%%%%%%%%%%%%%%%%%%%%%%%%%%%%%
%%%%%%%%%%%%%%%%%%%%%%%%%%%%%%%%%%%%%%%%%%%%%%%%%%%%%%%%%%%%%%%%%%%%%%
\section{Computation of the mass and of \texorpdfstring{$\C_{\beta,\gamma,p}^\star$}{the best radial constant}}\label{Appendix:Mass}

We shall denote by
\[
\sd:=|\S^{d-1}|=\frac{2\,\pi^{d/2}}{\Gamma(d/2)}
\]
the volume of the unit sphere $\S^{d-1}\subset\R^d$, for any integer $d\ge2$.

The mass of the Barenblatt stationary solution $\mathfrak B_{\beta,\gamma}(x)=\(C_M+|x|^{2+\beta-\gamma}\)^\frac1{m-1}$ is given by the identity
\[
M=\iRd{\(C_M+|x|^{2+\beta-\gamma}\)^\frac1{m-1}}=C_M^{\frac1{m-1}+\frac{d-\gamma}{2+\beta-\gamma}}\iRd{\(1+|x|^{2+\beta-\gamma}\)^\frac1{m-1}}
\]
and, using the change of variables $s=r^\alpha$, we obtain
\begin{multline*}
M_\star=\iRd{\(1+|x|^{2+\beta-\gamma}\)^\frac1{m-1}}=\sd\int_0^\infty\(1+r^{2+\beta-\gamma}\)^\frac1{m-1}\,r^{d-\gamma-1}\,dr\\
=\frac{\sd}\alpha\int_0^\infty\(1+s^2\)^\frac1{m-1}\,s^{n-1}\,ds=\frac{\sd}\alpha\,\frac{\Gamma(\frac n2)\,\Gamma(\frac1{1-m}-\frac n2)}{2\,\Gamma(\frac1{1-m})}
\end{multline*}
where $\alpha$ and $n$ are given in terms of $d$, $\beta$ and $\gamma$ by~\eqref{Eqn:alpha-n}. We recall that
\[
\int_0^\infty s^{2a-1}\,\(1+s^2\)^{-b}\,ds=\frac{\Gamma(a)\,\Gamma(b-a)}{2\,\Gamma(b)}\,.
\]
This allows us to compute $C_M$ in terms of $M/M_\star$:
\[
C_M=\(\frac M{M_\star}\)^\mu\quad\mbox{with}\quad\frac1\mu=\frac1{m-1}+\frac{d-\gamma}{2+\beta-\gamma}\,.
\]

With $\zeta$ given by~\eqref{zeta}, the constant $\C_{\beta,\gamma,p}^\star$ can be computed using the relation
\[
\C_{\beta,\gamma,p}^\star=\alpha^\zeta\,\mathsf K_{\alpha,n,p}^\star\quad\mbox{with}\quad\mathsf K_{\alpha,n,p}^\star:=\frac{\nrm{v_\star}{2p,d-n}}{\nrm{\D v_\star}{2,d-n}^\vartheta\,\nrm{v_\star}{p+1,d-n}^{1-\vartheta}}\,.
\]
as in Proposition~\ref{Prop:GNweigthed}. Using the computations of Section~\ref{Sec:VarCKN-LinearStab}, we obtain that
\[
\frac 1{\mathsf K_{\alpha,n,p}^\star}=\alpha^\vartheta\(\tfrac{4\,n}{p-1}\,\tfrac1{n+2-p\,(n-2)}\)^\frac\vartheta2\(\tfrac{2\,(p+1)}{n+2-p\,(n-2)}\)^\frac\vartheta{p+1}\(\sigma_d\,\tfrac{\Gamma(\frac n2)\,\Gamma(\frac{2p}{p-1}-\frac n2)}{2\,\Gamma(\frac{2p}{p-1})}\)^\zeta\,.
\]

%%%%%%%%%%%%%%%%%%%%%%%%%%%%%%%%%%%%%%%%%%%%%%%%%%%%%%%%%%%%%%%%%%%%%%
%%%%%%%%%%%%%%%%%%%%%%%%%%%%%%%%%%%%%%%%%%%%%%%%%%%%%%%%%%%%%%%%%%%%%%
\section{Some additional spectral properties}\label{Sec:SpectrumAditional}

This appendix collects some additional properties of the lowest eigenvalues and of the corresponding eigenfunctions. It completes the picture of Lemma~\ref{Lem:SpectrumResults}.

First of all, the function $f_{0,\kell}(s)=s^\eta$ solves~\eqref{EV} with $\kell=1$ if and only if
\[
\alpha^2\,\eta\,(\eta+n-2)=d-1\quad\mbox{and}\quad\Lambda_{0,\kell}=2\,\alpha^2\,\delta\,\eta\,.
\]
The unique positive solution $\eta$ is given by~\eqref{Eqn:eq-eta}. For later purpose, let us define $h(t):=t\,(t+n-2)-(d-1)/\alpha^2$ and observe that $h$ is increasing for $t>0$. The above equation for $\eta$ can be simply written as $h(\eta)=0$ and it has a unique positive solution.

We may also wonder if the function $f(s)=s^\eta$ with $\eta=(\beta+1)/\alpha$ is an eigenfunction, say $f_{\ellk,\kell}(s)$, for some $\ellk$, $\kell\in\N$, because we have $\frac d{dt}\int_{\R^d}x\,|x|^\beta\,u\,\frac{dx}{|x|^\gamma}=0$ if $u$ solves~\eqref{FD}. This moment conservation can be reinterpreted in terms of translations of the Barenblatt functions when $(\beta,\gamma)=(0,0)$ and, as observed in Section~\ref{Sec:Conclusion}, the corresponding invariance generates the eigenspace associated with $\lambda_{0,1}$. hence the question is to decide if something similar occurs when $(\beta,\gamma)\neq(0,0)$, although the presence of weights makes an interpretation in terms of invariances more delicate.

Solving~\eqref{EV} with $f(s)=s^\eta$ and $\eta=(\beta+1)/\alpha$ means that $\mu_\kell=\alpha^2\,\eta\,(\eta+n-2)=(\beta+1)\,(d-1)$, for some $\kell\in\N$. The unique solution corresponds to $k=0$ and is determined by
\[
\beta=\frac{\mu_\kell}{d-1}-1=\frac{\kell\,(\kell+d-2)}{d-1}-1=\frac{(\kell-1)\,(\kell+d-1)}{d-1}\,.
\]
Notice that we recover that only $\beta=0$ is eligible if $\kell=1$.

The pattern shown in Fig.~\ref{Fig:F4} is not generic, and three cases may occur: see Fig.~\ref{Fig:F5}. It depends on the choice of $\alpha$, $n$ and~$p$ as shown by the following elementary properties of the lowest eigenvalues:
%---------------------------------------------------------------------
\begin{figure}[ht]
\begin{center}
\includegraphics[width=4cm]{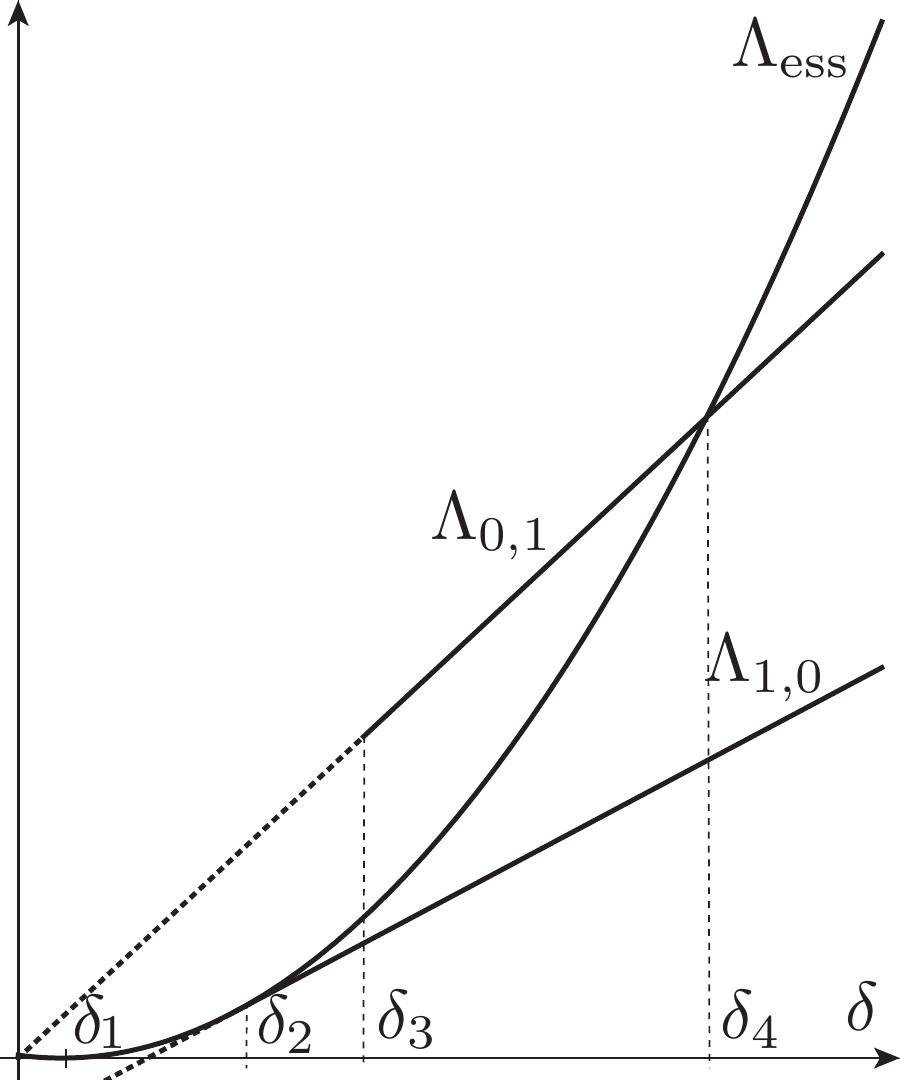}\hspace*{6pt}\includegraphics[width=4cm]{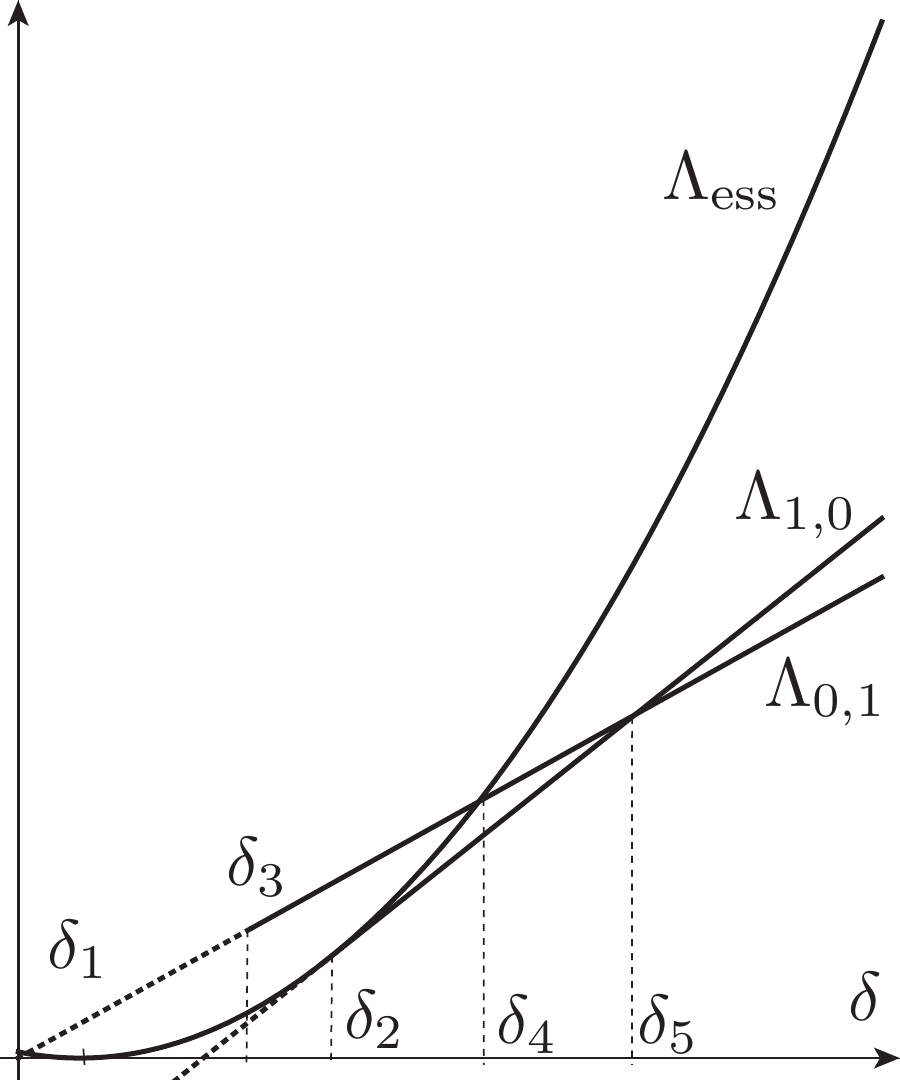}\hspace*{6pt}\includegraphics[width=4cm]{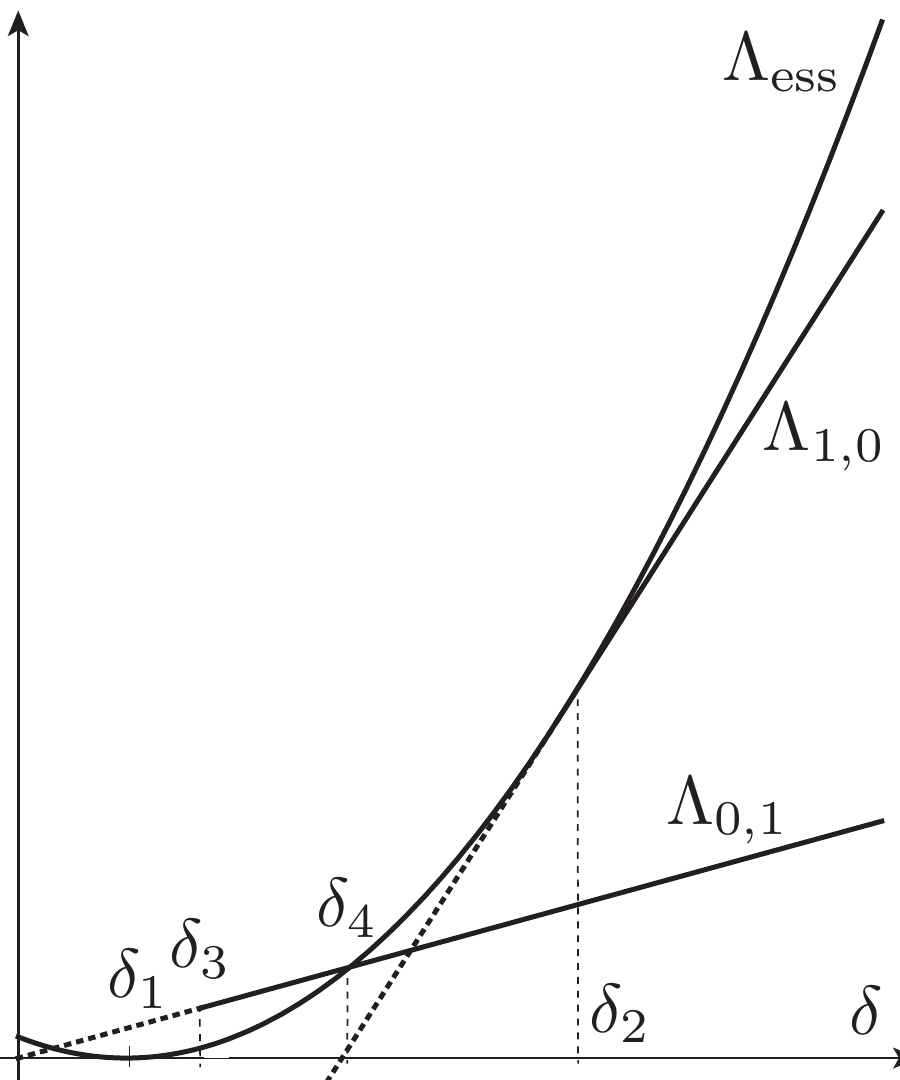}
\caption{\label{Fig:F5} The spectral gap and the lowest eigenvalues of $\mathcal L$ for $n=3$. The parabola represents $\Lambda_{\rm ess}$ as a function of $\delta$, $\Lambda_{1,0}$ is tangent to the parabola and $\Lambda_{0,1}$ is shown for $\eta=3.5$ (left), $\eta=1.4$ (center) and $\eta=0.35$ (right). The eigenvalues $\Lambda_{1,0}$ and $\Lambda_{0,1}$ are represented by a plain line only if the corresponding eigenvalues are in the space $\mathrm L^2(\R^d,\mathcal B^{2-m}\,|x|^{n-d}\,dx)$.}
\end{center}
\end{figure}
%---------------------------------------------------------------------
\begin{enumerate}
\item[(i)] $\Lambda_{\rm ess}=\tfrac14\,\alpha^2\,\(n-2-2\,\delta\)^2$ and, as a consequence, $\Lambda_{\rm ess}=0$ if and only if $\delta=\delta_1$ with
\[
\delta_1:=\frac{n-2}2\,.
\]
\item[(ii)] The eigenfunction $f_{1,0}(s)=s^2-\frac n{2\,\delta-n}$ associated with $\Lambda_{1,0}=2\,\alpha^2\,(2\,\delta-n)$ belongs to $\mathrm L^2(\R^d,\mathcal B^{2-m}\,|x|^{-\gamma}\,dx)$ if and only if $\delta>\delta_2$ with
\[
\delta_2:=\frac{n+2}2\,.
\]
\item[(iii)] The eigenfunction $f_{0,1}(s)=s^\eta$ associated with $\Lambda_{0,1}=\,2\,\alpha^2\,\delta\,\eta$ belongs to $\mathrm L^2(\R^d,\mathcal B^{2-m}\,|x|^{n-d}\,dx)$ if and only if $\delta>\delta_3$ with
\[
\delta_3:=\eta+\frac{n-2}2=\sqrt{\tfrac{d-1}\alpha+\(\tfrac{n-2}2\)^2}\,.
\]
It is clear that $\delta_3>\delta_1$. We also have $\delta_3<\delta_4$ where $\delta=\delta_4$ is determined by the condition $\Lambda_{0,1}=\Lambda_{\rm ess}$. After an elementary computation, we find indeed that
\[
\delta_4:=\eta+\frac{n-2}2+\frac{\sqrt{d-1}}\alpha\,.
\]
\item[(iv)] If $\eta\ge2$, \emph{i.e.}, if $h(2)\le0$, $\Lambda_{1,0}$ does not intersect with $\Lambda_{0,1}$ for any $\delta>0$. To observe an intersection of $\Lambda_{1,0}$ with $\Lambda_{0,1}$ in the range $\delta>\delta_1$, we need that
\[
\alpha>\alpha_1:=\sqrt{\frac{d-1}{2\,n}}\,.
\]
In other words, if $\alpha>\alpha_1$, the spectral gap is $\Lambda=\min\{\Lambda_{1,0},\Lambda_{\rm ess}\}$ and it is achieved among radial functions.

Now let us consider the case $\alpha>\alpha_1$. The intersection of $\Lambda_{1,0}$ with $\Lambda_{0,1}$ occurs for $\delta=\delta_5>\delta_2$ with 
\[
\delta_5:=\frac n{2-\eta}\,,
\]
if $\eta>\frac4{n+2}$, \emph{i.e.}, if $h(\frac4{n+2})<0$, which is equivalent to
\[
\alpha<\alpha_2:=\frac{n+2}{2\,n}\,\sqrt{d-1}\,.
\]
We observe that $\alpha_2>\alpha_1$ for any $n>0$. In the range $\alpha\in(\alpha_1,\alpha_2)$, by construction, we know that $\delta_2<\delta_4<\delta_5$, and the spectral gap is $\Lambda=\Lambda_{1,0}$ if $\delta\in(\delta_2,\delta_5]$ and $\Lambda=\Lambda_{0,1}$ for any $\delta\ge\delta_5$.

Finally, if $\alpha>\alpha_2$, we have $\delta_4<\delta_2$ and $\Lambda_{1,0}>\Lambda_{0,1}$ for any $\delta>\delta_2$. Moreover, in the range $\delta>\delta_4$, the spectral gap is $\Lambda=\Lambda_{0,1}$.

\item[(v)] Away from the symmetry breaking range, \emph{i.e.}, if $\alpha<\alpha_{\rm FS}$, we have $h(1)<0$, hence $\eta$, which is determined by $h(\eta)=0$, is larger than $1$.
\end{enumerate}

%%%%%%%%%%%%%%%%%%%%%%%%%%%%%%%%%%%%%%%%%%%%%%%%%%%%%%%%%%%%%%%%%%%%%%
%%%%%%%%%%%%%%%%%%%%%%%%%%%%%%%%%%%%%%%%%%%%%%%%%%%%%%%%%%%%%%%%%%%%%%
\section{Uniqueness of the radial optimal function}\label{Appendix:Uniqueness}

As a side result, we may observe that, up to a multiplication by a constant and a scaling, the optimal function $w_\star=\mathfrak B_{\beta,\gamma}^{m-1/2}$ is uniquely determined.
%---------------------------------------------------------------------
\begin{proposition}\label{Prop:Uniqueness} Assume that the parameters obey to~\eqref{parameters}. Then $w_\star$ is the unique radial optimal function for~\eqref{CKNrad} up to a multiplication by a constant and a scaling.\end{proposition}
%---------------------------------------------------------------------
\begin{proof} Optimality of $w$ in~\eqref{CKNrad} also means optimality of $w^{2p}$ in~\eqref{Ineq:E-EP}. By arguing as in~\cite{dolbeault:hal-01081098}, we obtain an improved version of the inequality which shows the existence of a remainder term proportional to $\mathcal F[w^{2p}]^2$. Hence $\mathcal F[w^{2p}]=0$, which is possible if and only if $w=w_\star$.\end{proof}

%%%%%%%%%%%%%%%%%%%%%%%%%%%%%%%%%%%%%%%%%%%%%%%%%%%%%%%%%%%%%%%%%%%%%%
%%%%%%%%%%%%%%%%%%%%%%%%%%%%%%%%%%%%%%%%%%%%%%%%%%%%%%%%%%%%%%%%%%%%%%
\section*{Acknowledgments} This research has been partially supported by the projects \emph{STAB} (J.D., B.N.) and \emph{Kibord} (J.D.) of the French National Research Agency (ANR). M.B.~has been funded by Project MTM2011-24696 and MTM2014-52240-P (Spain). This work has begun while M.B.~and M.M.~were visiting J.D.~and B.N.~in 2014. M.B.~thanks the University of Paris 1 for inviting him. M.M.~has been partially funded by the National Research Project ``Calculus of Variations'' (PRIN 2010-11, Italy) and by the ``Universit\`a Italo-Francese / Universit\'e Franco-Italienne'' (Bando Vinci 2013). J.D.~also thanks the University of Pavia for support.

\smallskip\noindent {\sl\small\copyright~2016 by the authors. This paper may be reproduced, in its entirety, for non-commercial purposes.}
%%%%%%%%%%%%%%%%%%%%%%%%%%%%%%%%%%%%%%%%%%%%%%%%%%%%%%%%%%%%%%%%%%%%%%
%%%%%%%%%%%%%%%%%%%%%%%%%%%%%%%%%%%%%%%%%%%%%%%%%%%%%%%%%%%%%%%%%%%%%%
%\bibliographystyle{AIMS}\bibliography{../BDMN}
\providecommand{\href}[2]{#2}
\providecommand{\arxiv}[1]{\href{http://arxiv.org/abs/#1}{arXiv:#1}}
\providecommand{\url}[1]{\texttt{#1}}
\providecommand{\urlprefix}{URL }

%%%%%%%%%%%%%%%%%%%%%%%%%%%%%%%%%%%%%%%%%%%%%%%%%%%%%%%%%%%%%%%%%%%%%%
%%%%%%%%%%%%%%%%%%%%%%%%%%%%%%%%%%%%%%%%%%%%%%%%%%%%%%%%%%%%%%%%%%%%%%
\end{document}